\documentclass[11pt,reqno]{amsart}

\usepackage{amsfonts}
\usepackage{amssymb}
\usepackage{amsmath}
\usepackage{amsthm}
\usepackage[utf8]{inputenc}
\usepackage[T1]{fontenc}
\usepackage{mathrsfs}
\usepackage{microtype}
\usepackage{xcolor}
\usepackage[caption=false]{subfig} 
\usepackage[colorlinks=true, urlcolor=blue]{hyperref}
\usepackage{booktabs} 
\usepackage{array}  
\usepackage{enumitem}
\usepackage{mathtools} 
\usepackage{graphicx}
\usepackage{subcaption}
\usepackage{tabularx}  
\usepackage{multirow}
\usepackage{fullpage}
\usepackage{pgfplots}
\pgfplotsset{compat=1.15}
\usetikzlibrary{arrows}
\usetikzlibrary{arrows.meta}
\usetikzlibrary{positioning,fit}
\usetikzlibrary{pgfplots.groupplots} 
\usetikzlibrary{external}
\tikzset{external/optimize=true}

\DisableLigatures{encoding = *, family = * }

\newtheorem{theorem}{Theorem}

\newcommand{\NO}[1]{\widehat{#1}}

\newcommand{\RR}{\mathbb{R}}
\newcommand{\CC}{\mathbb{C}}

\numberwithin{equation}{section}
\numberwithin{theorem}{section}

\title{Learning the Riccati solution operator for time-varying LQR via Deep Operator Networks}

\author{Jun Chen\textsuperscript{\,$\ast$}}  

\address{\textsuperscript{$\ast$}\, School of Mathematics and Statistics, Beijing Institute of Technology, 100081 Beijing, China.
	\newline \indent \textsuperscript{$\dagger$}\, Chair of Computational Mathematics, DeustoTech, University of Deusto, Avenida de las Universidades 24, 48007 Bilbao, Basque Country, Spain.}

\email{chenjun\_bcbm@163.com}
\thanks{This project has received funding from the European Research Council (ERC) under the European Union's Horizon 2030 research and innovation programme (grant agreement NO: 101096251-CoDeFeL). UB was partially supported by the Grant PID2023-146872OB-I00-DyCMaMod of MICIU (Spain) and by the COST Actions ``CA24122 - multiscale Stochastics, Patterns, and Analysis of Combinatorial Environments'' and ``CA24136 - Interactions between Control Theory and Machine Learning''. CJ and JW were supported by the National Natural Science Foundation of China under Grant 92471108 and 12131008, and the Program of China Scholarship Council(Grant No.202506030033)}

\author{Umberto Biccari\textsuperscript{\,$\dagger$}}
\email{umberto.biccari@deusto.es}

\author{Junmin Wang\textsuperscript{\,$\ast$}}
\email{jmwang@bit.edu.cn}

\begin{document}

\begin{abstract}
We propose a computational framework for replacing the repeated numerical solution of differential Riccati equations in finite-horizon Linear Quadratic Regulator (LQR) problems by a learned operator surrogate. Instead of solving a nonlinear matrix-valued differential equation for each new system instance, we construct offline an approximation of the associated solution operator mapping time-dependent system parameters to the Riccati trajectory. The resulting model enables fast online evaluation of approximate optimal feedback laws across a wide class of systems, thereby shifting the computational burden from repeated numerical integration to a one-time learning stage.

From a theoretical perspective, we establish control-theoretic guarantees for this operator-based approximation framework. In particular, we derive bounds quantifying how operator approximation errors propagate to feedback performance, trajectory accuracy, and cost suboptimality, and we prove that exponential stability of the closed-loop system is preserved under sufficiently accurate operator approximation. These results provide a systematic framework to assess the reliability of data-driven approximations in optimal control.

On the computational side, we design tailored DeepONet architectures for matrix-valued, time-dependent problems and introduce a progressive learning strategy to address scalability with respect to the system dimension. Numerical experiments on both time-invariant and time-varying LQR problems demonstrate that the proposed approach achieves high accuracy and strong generalization across a wide range of system configurations, while delivering substantial computational speedups compared to classical Riccati solvers. The method offers an effective and scalable alternative for parametric and real-time optimal control applications.

More broadly, the proposed approach illustrates how operator learning can be used to build reusable surrogates for nonlinear matrix differential equations arising in control and dynamical systems.

\end{abstract}

\maketitle  
\section{Introduction and motivations}

Optimal control plays a central role in science and engineering, offering a systematic framework for designing control strategies that regulate dynamical systems while balancing performance and control effort \cite{kirk2004optimal,kwakernaak1972linear,sontag2013mathematical,trelat2005controle}.

Among the most widely used formulations, the Linear Quadratic Regulator (LQR) provides a tractable and robust methodology, with applications including robotics, aerospace engineering, signal processing, and autonomous systems \cite{alcala2018autonomous,dhingra2025modeling,elkhatem2022robust,gu2006generalized,peng2025linear,su2024lqr}. Its formulation is characterized by the Riccati equation, which determines the optimal feedback controller.

Despite its classical nature, solving LQR problems remains computationally demanding in modern settings where system parameters vary or must be evaluated repeatedly. In the case of linear time-varying systems, the optimal feedback law is characterized by the Differential Riccati Equation (DRE), a nonlinear matrix-valued differential equation that must be solved backward in time. Classical numerical methods for Riccati equations include direct and iterative schemes based on matrix factorizations, invariant subspace techniques, and backward integration methods, whose computational cost typically scales cubically with the state dimension \cite{bittanti1991riccati,lancaster1995algebraic}.

While this approach is effective for solving a single instance of an LQR problem, it becomes a major limitation in settings where one must evaluate optimal feedback laws for a large number of system configurations. Such situations arise naturally in parametric control, uncertainty quantification, and real-time applications like Model Predictive Control (MPC), where optimal control problems must be solved sequentially in real time for evolving system states and parameters. In these contexts, each new parameter instance requires solving a Riccati equation from scratch, leading to a computational cost that grows linearly with the number of instances and rapidly becomes prohibitive in high-dimensional settings.

This computational bottleneck motivates the search for alternative paradigms that can exploit the structure of the problem across different parameter configurations. In particular, the repeated solution of Riccati equations can be replaced by the evaluation of a surrogate model that directly maps system parameters \textendash\, such as time-dependent matrices $(A(t), B(t), Q(t), R(t))$ \textendash\, to the associated solution. 

Under this perspective, in this work, we adopt an Operator Learning (OL) approach in which the Riccati equation is viewed as a nonlinear mapping between function spaces. Instead of solving the DRE for each new instance, we aim to approximate this solution operator once, and then reuse it to enable fast evaluation of approximate optimal feedback laws for previously unseen systems. This leads to an offline-online computational strategy, where the main computational effort is performed during a training stage, and the evaluation for new parameter configurations becomes essentially instantaneous.

To realize this operator-based framework, we employ Deep Operator Networks (DeepONets \cite{lanthaler2022error,lu2021learning}), which are specifically designed to learn mappings between infinite-dimensional function spaces. DeepONets combine a branch network that encodes the input functions with a trunk network that represents the output dependence, providing a flexible architecture for approximating operators arising in differential equations. This enables a single trained model to generalize across a broad class of systems with varying dynamics and cost structures. In contrast to existing learning-based approaches for solving differential equations and control problems \textendash, such as Physics-Informed Neural Networks (PINNs) \cite{quero2024physics,raissi2019physics}, originally developed for PDE approximation, and Reinforcement Learning (RL) methods \cite{cui2025learning,yu2026stochastic} \textendash, which generally require instance-wise training or optimization, our approach targets the full Riccati solution operator, allowing a single trained model to generalize across a broad class of systems.

More precisely, existing approaches differ fundamentally in how they treat parametric variability. PINNs are typically trained to approximate the solution of a single differential equation instance by enforcing the governing equations in the loss function. As a consequence, each new system configuration requires retraining or fine-tuning, which limits their applicability in settings involving large families of parameter-dependent problems.

RL methods, on the other hand, aim to learn control policies through interaction with the system, often without explicitly exploiting the underlying Riccati structure. While RL provides a flexible framework for control, it generally focuses on policy approximation rather than on learning the mapping from system parameters to optimal solutions, and may require extensive sampling or exploration for each new environment.

In contrast, the OL approach adopted in this work targets directly the parametric solution map associated with the Riccati equation. By learning this mapping offline, the resulting model can be reused across a wide range of system configurations without retraining, enabling an amortized computational strategy that fundamentally differs from instance-wise approaches.

While we focus on the Riccati equation arising in LQR problems, the proposed approach is representative of a broader computational paradigm: learning reusable solution operators for nonlinear matrix-valued differential equations in control and dynamical systems. In this perspective, the Riccati equation serves as a canonical testbed where both numerical efficiency and control-theoretic properties, such as stability and optimality, can be rigorously assessed.

A key challenge in this setting is to ensure that the learned approximation preserves the fundamental properties of the optimal solution, in particular stability and performance guarantees of the resulting closed-loop system. To address this issue, we establish a rigorous link between the approximation error of the learned Riccati operator and the behavior of the induced feedback law. More precisely, we show that errors at the operator level propagate in a controlled manner to the feedback gain, the closed-loop trajectories, and the associated cost functional. In particular, the performance loss of the learned controller can be explicitly bounded in terms of the operator approximation error, and exponential stability of the closed-loop system is preserved provided this error remains sufficiently small. This establishes a complete error propagation chain from operator approximation to control performance, providing a control-theoretic interpretation of OL in this setting.

The proposed framework is validated through numerical experiments demonstrating that the learned operator achieves high accuracy and near-perfect stabilization rates, while significantly reducing computational cost compared to classical solvers. Furthermore, we show that the approach scales to higher-dimensional systems and can exploit structural similarities across dimensions through a progressive learning strategy, whose empirical performance suggests the presence of low-dimensional structure in the underlying operator.

\smallskip
\noindent The main contributions of this work can be summarized as follows:
\begin{itemize}
    \item[1.] we introduce an operator-based computational framework for LQR, replacing repeated numerical solution of Riccati equations by a learned solution operator;
    \item[2.] we design a DeepONet-based architecture tailored to matrix-valued, time-dependent control problems;
    \item[3.] we establish a control-theoretic error propagation framework linking operator approximation error to feedback gain accuracy, trajectory deviation, cost suboptimality, and closed-loop stability;
    \item[4.] we demonstrate the effectiveness and scalability of the approach through extensive numerical experiments.
    \end{itemize}

The remainder of the paper is organized as follows. Section \ref{sec:problem} introduces the finite-horizon LQR problem and formulates the associated Riccati equation from an OL perspective. Section \ref{sec:deeponet} presents the DeepONet-based approximation framework, including the network architecture, data generation strategies, and a progressive learning approach to address scalability across dimensions. Section \ref{sec:math} provides theoretical guarantees, including approximation results for the Riccati operator and an analysis of error propagation and closed-loop stability. Section \ref{sec:experiments} reports numerical experiments illustrating the accuracy, stability, and computational efficiency of the proposed method in both time-invariant and time-varying settings. Finally, Section \ref{sec:conclusions} gathers our conclusions and propose new future research directions.

\section{Problem formulation and operator viewpoint}\label{sec:problem}

From an engineering perspective, LQR provides a systematic framework to design feedback controllers that achieve an optimal trade-off between system performance and control effort. More precisely, the objective of LQR is to regulate the state of a dynamical system toward a desired equilibrium (typically the origin) while minimizing a quadratic cost functional that penalizes both deviations of the state and the magnitude of the control input.

LQR has been extensively used across a wide range of engineering domains. In aerospace engineering, it is employed for attitude control and trajectory stabilization of aircraft and spacecraft, where maintaining stability under perturbations is critical \cite{choi1999lqr,stevens2015aircraft}. In robotics, LQR-based controllers are used for motion planning, balancing problems (e.g., inverted pendulum or humanoid robots) \cite{su2024lqr,klemm2020lqr}, and trajectory tracking \cite{kumar2013robust}. In electrical and mechanical systems, it appears in vibration suppression, power system regulation, and process control \cite{mohammed2018active,singh2015decentralized,slimane2025deep}. More recently, LQR also plays a central role in modern control architectures such as MPC, where it serves as a local optimal controller or terminal cost approximation \cite{kalabic2020constraint,he2016dual}.

In this work, we consider the finite-horizon LQR problem for continuous-time linear time-varying systems. Let $T > 0$ be a fixed time horizon. The system dynamics are given by
\begin{equation}\label{eq:state}
    \begin{cases}
        \dot{x}(t) = A(t)x(t) + B(t)u(t),\quad t\in[0,T],
        \\
        x(0) = x_0,
    \end{cases}
\end{equation}
where $x(t) \in\RR^n$ denotes the state and $u(t) \in\RR^m$ with $m\leq n$ the control input. The matrices $A(\cdot)\in\RR^{n\times n}$ and $B(\cdot)\in\RR^{n\times m}$ are assumed to be continuous in time.

In this context of time-varying systems, the finite-horizon LQR problem is particularly relevant for transient control tasks, where system parameters evolve over time or where control objectives are defined over a finite time interval. Examples include trajectory optimization \cite{reddy2024learning}, time-dependent resource allocation \cite{xu2018optimal}, and adaptive control in uncertain environments \cite{fong2018dual}.

\smallskip
\noindent The LQR controller is designed by minimizing the quadratic cost functional
\begin{equation}\label{eq:functional}
    J(u) = \int_0^T \Big(x(t)^\top Q(t)x(t) + u(t)^\top R(t)u(t) \Big)\, dt + x(T)^\top P_T x(T),
\end{equation}
where $Q(\cdot)\in\RR^{n\times n}$ and $R(\cdot)\in\RR^{m\times m}$ are continuous and symmetric weighting matrices, with $Q(\cdot)$ positive semi-definite and $R(\cdot)$ positive definite, and $P_T\in\RR^{n\times n}$ is a positive definite and time-invariant terminal cost matrix.

The structure of this cost functional \eqref{eq:functional} reflects two competing design requirements. On the one hand, the term $x^\top Q x$ enforces state regulation, promoting accuracy, stability, and tracking performance. On the other hand, the term $u^\top R u$ penalizes control effort, preventing excessive actuator usage, energy consumption, or aggressive control actions that may be infeasible in practice. The choice of the weighting matrices $Q$, $R$, and $P_T$ therefore encodes engineering specifications such as precision, robustness, and resource constraints.

One of the key advantages of the LQR framework is that it yields a closed-loop linear state feedback law that is optimal and stabilizing under mild assumptions. This makes it particularly attractive in real-time applications, where computational simplicity and robustness are essential. In more detail, it is well known \cite{kwakernaak1972linear} that, under standard stabilizability and detectability assumptions (see Section \ref{sec:math}), the optimal control is given in feedback form
\begin{align*}
    u^\ast(t)=-K^\ast(t)x(t), \quad\text{ with } K^\ast(t)=R^{-1}(t)B^\top(t)P^\ast(t),    
\end{align*}
where $P^\ast(t)$ is the unique symmetric positive definite solution of the Differential Riccati Equation (DRE)
\begin{equation}\label{eq:Riccati}
    \begin{cases}
        -\dot{P}(t) = A(t)^\top P(t) + P(t)A(t) - P(t)B(t)R^{-1}(t)B(t)^\top P(t) + Q(t), \; t\in[0,T]
        \\[2pt]
        P(T) = P_T.
    \end{cases}
\end{equation}

Despite its conceptual simplicity, the practical deployment of LQR in modern applications is often hindered by computational challenges. As discussed in the introduction, solving the Riccati equation \eqref{eq:Riccati} repeatedly for varying system parameters can become prohibitively expensive in high-dimensional or real-time settings. To overcome this bottleneck, in this work, we adopt an operator viewpoint. We denote by
\begin{align*}
    \Theta = (A(\cdot), B(\cdot), Q(\cdot), R(\cdot), P_T)
\end{align*}
the collection of system parameters, and we define the mapping
\begin{align*}
    \mathcal F : \Theta \mapsto P(\cdot),    
\end{align*}
which associates to each parameter instance the corresponding solution of the Riccati equation.

From a computational viewpoint, this operator encapsulates the full dependence of the optimal control law on the system parameters. Classical approaches evaluate this mapping by solving the Riccati equation \eqref{eq:Riccati} numerically for each input $\Theta$. In contrast, our objective is to approximate this operator once and reuse it, enabling efficient evaluation across multiple problem instances.

\section{DeepONet-based approximation of the Riccati operator}\label{sec:deeponet}

We present a DeepONet-based framework to approximate the Riccati solution operator introduced in Section \ref{sec:problem}. The goal is to construct a neural network that maps system parameters directly to the corresponding solution of the Riccati equation, enabling fast evaluation for new instances. 

\subsection{Operator Learning formulation: the DeepONet architecture}

Recall that we aim to approximate the mapping 
\begin{align*}
    \mathcal F : \Theta \mapsto P(\cdot),    
\end{align*}
which associates to each system configuration $\Theta = (A(\cdot), B(\cdot), Q(\cdot), R(\cdot), P_T)$ the solution of the corresponding Riccati equation.

In practice, both the input functions and the output trajectory must be represented in finite-dimensional form. To this end, as we will see, the time-dependent matrices are sampled or parametrized using suitable discretizations or basis expansions, leading to a finite-dimensional representation of the input. Similarly, the solution $P(t)$ is evaluated at a set of time points or represented through a functional basis.

Under this representation, the OL problem is reduced to a supervised learning task: given a dataset of input–output pairs $(\Theta_i, P_i)$, where $P_i$ is the solution of the Riccati equation corresponding to $\Theta_i$, the objective is to construct a model that approximates the mapping $\mathcal F$ and generalizes to unseen parameter instances.

To approximate $\mathcal F$, we employ a DeepONet, which is designed to learn mappings between function spaces. The architecture consists of two main components: a branch network and a trunk network. The branch network processes the input representation of the system parameters $\Theta$, encoding them into a finite-dimensional feature vector. The trunk network takes as input the time variable $t$ and outputs a set of basis functions that capture the temporal structure of the solution.

The final output is obtained by combining the outputs of the branch and trunk networks. More precisely, the predicted solution is represented as
\begin{align*}
    \mathcal G(\Theta)(t) = \sum_{k=1}^p \beta_k(\Theta)\tau_k(t) \sim P(t),
\end{align*}
where $\beta_k(\Theta)$ are coefficients produced by the branch network and $\tau_k(t)$ are basis functions produced by the trunk network.

In order to handle matrix-valued outputs, the network is designed so that the final layer produces a vector representation that can be reshaped into a matrix of appropriate dimension. The architecture can be adapted to different problem sizes by adjusting the number of basis functions and network widths.

The DeepONet architecture provides a flexible framework for approximating nonlinear mappings between function spaces, such as the Riccati solution map $\mathcal F$. Its theoretical foundation relies on universal approximation results for operators \cite{deng2022approximation,lu2021learning}, which ensure that sufficiently regular mappings can be approximated with arbitrary accuracy on compact subsets.

It is worth emphasizing that this OL formulation differs conceptually from classical neural PDE solvers and control learning approaches. In particular, rather than approximating a single solution (as in PINNs) or directly learning a control policy (as in RL), the objective here is to approximate the full mapping from system parameters to Riccati solutions. This distinction is crucial in applications where the same model must be evaluated repeatedly for different system configurations, as it enables a clear separation between an offline training phase and a fast online evaluation stage.

In the present setting, the mapping $\mathcal F$, which associates system parameters to the corresponding solution of the Riccati equation, is well-defined and depends continuously on the input data. Moreover, under standard assumptions on stabilizability and positive definiteness, this dependence is locally Lipschitz, as follows from classical sensitivity results for Riccati equations \cite{abou2012matrix,bittanti1991riccati}. In this context, we have by \cite[Theorem 1]{lu2021learning} that, for any $\varepsilon > 0$, there exists a DeepONet $\mathcal G$ such that for all $\Theta$ and $t \in [0,T]$,
\begin{align*}
    \big\| \mathcal F(\Theta)(t) - \mathcal G(\Theta)(t) \big\| < \varepsilon.
\end{align*}

\subsection{Data generation and training strategy}

The performance of the proposed OL approach, in particular the level of approximation $\varepsilon$ that DeepONet can achieve, depends critically on the construction of representative training data \cite{lanthaler2022error}. In this section, we describe how datasets are generated for both time-invariant and time-dependent systems, as well as the training procedure used to fit the model.

\subsubsection{Data generation for time-invariant parameters}

For time-invariant parameters $(A,B,Q,R)$, the Riccati equation \eqref{eq:Riccati} reduces to the Algebraic Riccati Equation (ARE)
\begin{align}\label{eq:RiccatiAlgebraic}
    A^\top P + PA - PBR^{-1}B^\top P + Q = 0.
\end{align}

In this case, our task reduces to learn the map $(A,B,Q,R)\mapsto P$. Notice that this is no longer an operator but rather a matrix-valued function. Nevertheless, the DeepONet methodology can still be applied.

To generate training samples, we shall consider structured families of linear systems $(A,B)$ and design matrices $(Q,R)$ that ensure controllability and cover a broad range of dynamical behaviors. 

For what concerns $Q$ and $R$, we will restrict them to be diagonal matrices with non-negative and positive diagonal entries, respectively, to preserve the semi-positive and positive nature required to build LQR controllers. This assumption simplifies the data generation procedure and allows for a controlled exploration of the influence of state and control weights. Notice that, while symmetric matrices are diagonalizable, restricting to diagonal structures does not cover the full generality of possible couplings between state components, and should therefore be interpreted as a modeling choice rather than a general representation.

For the dynamics matrices $A$ and $B$, instead, we adopt parametrizations based on the Brunovsky canonical form, which provides a complete characterization of controllable linear systems up to similarity transformations. This is based on a fundamental property of control systems \cite{zhou1996robust,kwakernaak1972linear}: for any controllable pair $(A,B)$, there exists an invertible change of coordinates under which the system can be decomposed into a block-diagonal structure consisting of independent controllable subsystems. More precisely, we can find a matrix $T\in\RR^{n\times n}$ such that
\begin{align*}
    (T^{-1} A T,\, T^{-1} B) = (A_c, B_c)
\end{align*}
with 
\begin{align}\label{eq:brunovsky1}
    A_c = \mathrm{diag}(A_{c,1}, \dots, A_{c,r}),\quad B_c = \mathrm{diag}(B_{c,1}, \dots, B_{c,r}),    
\end{align}
and where each $(A_{c,i}, B_{c,i})$ is a single-input controllable canonical block of size $n_i$, with
\begin{align}\label{eq:brunovsky2}
A_{c,i} = \begin{bmatrix}
    0 & 1 & & 0 \\
    & \ddots & \ddots & \\
    -a_{i,0} & \cdots & -a_{i,n_i-2} & -a_{i,n_i-1}
\end{bmatrix},\quad
B_{c,i} = \begin{bmatrix} 0 \\ \vdots \\ 1 \end{bmatrix},\quad n = \sum_{i=1}^r n_i.
\end{align}

This representation offers two key advantages for dataset generation. First, it ensures that all sampled systems are controllable by construction, which is a fundamental requirement for the well-posedness of the LQR problem and the existence of a unique positive definite solution to the Riccati equation. Second, it provides a systematic parametrization of the space of linear dynamics through the coefficients of the companion matrices, enabling controlled exploration of a wide range of dynamical behaviors.

From an approximation perspective, this parametrization induces a low-dimensional yet expressive representation of the space of admissible systems. Since similarity transformations do not affect the qualitative properties of the Riccati solution (e.g., stabilizability and optimal cost structure), learning the operator in Brunovsky coordinates captures the essential features of the mapping $(A,B,Q,R)\mapsto P$ without redundancy. This allows the network to focus on intrinsic dynamical features rather than superficial coordinate variations.

Overall, the use of Brunovsky canonical forms leads to a dataset that is both structurally rich and mathematically controlled, ensuring coverage of representative system behaviors while maintaining consistency with the theoretical assumptions underlying the Riccati equation.

\subsubsection{Data generation for time-dependent parameters} 

For time-dependent parameters, the previous approach based on the Brunovsky decomposition is no longer applicable. We shall then define a different strategy to generate the matrices $A(t)$, $B(t)$, $Q(t)$ and $R(t)$. 

As in the time-independent case before, the dataset must be representative enough so to ensure good generalization performances for DeepONet once trained. To adhere to this principle, the dataset is constructed using a truncated trigonometric expansion of the system matrices. In particular, we will consider matrices of the form
\begin{equation}\label{eq:trigonometric}
    \begin{array}{l}
        \displaystyle A(t) = A_0 + \sum_{i=1}^{r_{\text{base}}} \Big(C_{1i} \phi_i(t) + C_{2i} \psi_i(t)\Big),
        \\
        \displaystyle B(t) = B_0 + \sum_{i=1}^{r_{\text{base}}} \Big(D_{1i} \phi_i(t)+ D_{2i} \psi_i(t)\Big), 
        \\
        \displaystyle Q(t) = M^\top(t) M(t) \quad\text{ with }\quad  M(t) = M_0 + \sum_{i=1}^{r_{\text{base}}} \Big(E_{1i} \phi_i(t)+ E_{2i} \psi_i(t)\Big),     
    \end{array}
\end{equation}
with randomly sampled coefficients 
\begin{displaymath}
    \begin{array}{lll}
        A_0\in\RR^{n\times n}, & B_0\in\RR^{n\times m}, & M_0 \in\RR^{n\times n}
        \\
        C_{ki}\in\RR^{n\times n}, & D_{ki}\in\RR^{n\times m}, & E_{ki}\in\RR^{n\times n}    
    \end{array}, \quad i= 1,\dots, r_{\text{base}},\; k=1,2,
\end{displaymath}
and where we have denoted 
\begin{align*}
    \phi_i(t) = \cos(i\pi t) \quad\text{ and }\quad \psi_i(t)=\sin(i\pi t).    
\end{align*}
The control weighting matrix $R(t)$, instead, will be assumed to be the identity, i.e., $R(t) = R = I_m$. 

This construction of the data can be interpreted as sampling from a structured low-frequency manifold in function space induced by truncated Fourier expansions. In this perspective, the matrices $A(\cdot)$, $B(\cdot)$ and $Q(\cdot)$ are not arbitrary time-dependent functions, but belong to a finite-dimensional subspace spanned by smooth basis functions. This aligns with the OL framework, where the objective is to approximate mappings between compact subsets of infinite-dimensional spaces.

From an approximation-theoretic viewpoint, Fourier-type bases provide dense representations of smooth functions, ensuring that the sampled trajectories capture a wide range of temporal behaviors while remaining controlled in regularity. The truncation level $r_{\text{base}}$ thus defines the intrinsic dimension of the function class and acts as a complexity parameter for the input space. 

From a control perspective, this construction ensures that the generated systems exhibit bounded variation and regular dynamics, which guarantees well-posedness of the differential Riccati equation and stability of the numerical solvers. As a consequence, the dataset can be viewed as a compact, low-frequency subset of admissible system trajectories, on which the Riccati operator is learned.

Finally, we emphasize that this choice is not restrictive: by increasing the number of basis functions, one can approximate increasingly complex time dependencies. Therefore, the proposed dataset construction provides a systematic and scalable framework for generating representative training samples for time-varying LQR problems. In this sense, the trigonometric basis acts as a bridge between infinite-dimensional function spaces and their finite-dimensional representations used in the learning framework.

\subsubsection{Training procedure}

The model is trained using supervised learning, where the objective is to minimize the discrepancy between the predicted and true Riccati solutions over the training dataset. The loss function we employed is Mean Square Error (MSE), defined as
\begin{displaymath}
    \begin{array}{ll}
        \displaystyle\mathcal L = \frac{1}{N} \sum_{i=1}^N \|P_i^{\text{pred}} - P_i^{\text{true}}\|_F^2, & \quad\text{time-independent parameters}
        \\
        \displaystyle \mathcal L = \frac{1}{NM} \sum_{i=1}^N\sum_{m=1}^M \|P_i^{\text{pred}}(t_m) - P_i^{\text{true}}(t_m)\|_F^2, & \quad\text{time-dependent parameters}
    \end{array}
\end{displaymath}
where $P_i^{\text{true}}$ denotes the reference solution, $P_i^{\text{pred}}$ the network prediction, and $\{t_m\}_{m=1}^N$ are the discretization points of the time interval $[0,T]$.

Training is performed using standard optimization algorithms (e.g., Adam), with a suitable choice of learning rate, batch size, and number of epochs. The dataset is split into training and testing subsets to evaluate generalization performance. Complete implementation details will be given in Section \ref{sec:experiments}.

\subsection{Progressive Operator Learning across dimensions}\label{subsec:progressive}

A central challenge in OL is scalability with respect to the system dimension. In our case, as the state dimension increases, both the inputs $(A,B,Q,R)$ and the output $P$ grow quadratically, making direct learning of the Riccati operator increasingly expensive in terms of data, model size, and training time.

To address this issue, we propose a \textbf{Progressive Operator Learning} strategy, aimed at transferring knowledge learned in low-dimensional settings to higher-dimensional problems. The key idea is that the Riccati solution operator may exhibit structural similarities across dimensions, so that part of its complexity can be captured in a lower-dimensional representation and reused when moving to larger systems.

This intuition is partly supported by structural considerations in the time-invariant setting. In this case, controllable linear systems admit a canonical representation via the Brunovsky canonical form, under which the dynamics decompose into chains of integrators. The associated Riccati solutions inherit a hierarchical structure reflecting these subsystems. From this perspective, increasing the system dimension can often be interpreted as extending an existing structure rather than introducing entirely new dynamics, which provides a rationale for transferring a learned operator across dimensions.

However, this argument relies on the availability of such canonical decompositions and therefore applies primarily to the time-invariant case. In the time-dependent setting, no analogous global normal form is available, and the system matrices may exhibit more complex temporal variations. As a consequence, the assumption that the Riccati operator admits a dimension-robust representation cannot be justified theoretically in the same way.

For this reason, the proposed progressive strategy should be understood as an empirically motivated approach aimed at exploiting potential low-dimensional structure in the Riccati operator across system dimensions. We do not claim the existence of a general dimension-independent representation of the Riccati operator. Instead, the construction can be viewed as a form of latent operator factorization, where the core mapping is captured in a reduced representation and complemented by dimension-specific transformations. The validity of this interpretation is ultimately assessed through numerical experiments presented in Section \ref{sec:experiments}, which indicate that such a reduced representation can capture a significant portion of the operator complexity in the considered regimes.

Let $\mathcal{F}_n$ denote the operator mapping system parameters to the solution of the Riccati equation in dimension $n$. Our objective is to approximate $\mathcal{F}_{n_1}$ for $n_1 > n$ by leveraging a pre-trained approximation $\NO{\mathcal{F}}_n$ of $\mathcal{F}_n$. This is achieved by constructing a composite architecture of the form
\begin{align*}
    \NO{\mathcal{F}}_{n_1} \sim \mathcal{E}_{n_1} \circ \NO{\mathcal{F}}_n \circ \mathcal{I}_{n_1},
\end{align*}
where
\begin{itemize}
    \item $\mathcal{I}_{n_1}$ is an embedding operator mapping high-dimensional inputs into a lower-dimensional latent space;
    \item $\mathcal{E}_{n_1}$ is a lifting operator reconstructing the high-dimensional output.
\end{itemize}

This decomposition separates the learning task into a shared operator representation (captured by $\NO{\mathcal{F}}_n$) and dimension-specific adaptations (captured by $\mathcal{I}_{n_1}$ and $\mathcal{E}_{n_1}$), allowing one to significantly reduce the number of trainable parameters and the associated computational cost. 

From this viewpoint, the progressive architecture can be interpreted as approximating the Riccati operator through a structured low-dimensional latent representation, with embedding and lifting layers acting as dimension-dependent encoders and decoders.

\subsubsection{Architecture and error decomposition}

In practice, we will implement the progressive strategy by combining a pre-trained DeepONet with two additional neural components:
\begin{itemize}
    \item[] \textbf{Embedding layer.} A fully connected network $\mathcal{I}_{n_1}$ that maps the flattened high-dimensional input (namely, concatenated entries of $(A,B,Q,R)$) into a feature vector compatible with the input space of the pre-trained low-dimensional branch network.
    \item[] \textbf{Extension} (\textbf{lifting}) \textbf{layer.} A fully connected network $\mathcal{E}_{n_1}$ that maps the low-dimensional output of the pre-trained model to the higher-dimensional Riccati solution.
\end{itemize}

The parameters of the pre-trained DeepONet are kept fixed, while only the embedding and extension layers are trained on high-dimensional data. This significantly reduces the number of trainable parameters and the associated computational cost.

Moreover, this progressive construction induces a natural decomposition of the approximation error. Denoting by $\mathcal{F}_{n_1}$ the true operator and by $\NO{\mathcal{F}}_{n_1}$ its progressive approximation, we can formally write
\begin{align*}
    \|\mathcal{F}_{n_1}(\Theta) - \NO{\mathcal{F}}_{n_1}(\Theta)\| \leq \|\mathcal{F}_{n_1}(\Theta) - \mathcal{E}_{n_1} \circ \mathcal{F}_n \circ \mathcal{I}_{n_1}(\Theta)\| + \|\mathcal{F}_n - \NO{\mathcal{F}}_n\| + \varepsilon_\text{adapt},
\end{align*}
where the three terms correspond respectively to:
\begin{itemize}
    \item[1.] a \textbf{representation error}, measuring how well the high-dimensional operator can be expressed through the low-dimensional one;
    \item[2.] the \textbf{approximation error} of the pre-trained DeepONet;
    \item[3.] an \textbf{adaptation error}, arising from the training of the embedding and lifting networks.
\end{itemize}

This decomposition highlights that the effectiveness of the progressive strategy depends on the existence of a suitable low-dimensional representation and on the quality of the learned transformations.

\section{Theoretical guarantees for the learned Riccati operator}\label{sec:math}

The goal of this section is to characterize how approximation errors at the level of the Riccati solution operator affect the resulting control law and closed-loop dynamics. In particular, we aim to quantify how perturbations in the operator propagate through the feedback gain to the state trajectories and the associated cost functional. This provides a systematic framework to assess the reliability of learned operators in control applications.

\subsection{Notation}
Denote by $\mathbb{S}^n$ the set of $n\times n$ real symmetric matrices. Let $\mathbb{S}^n_+ \subset \mathbb{S}^n$ and $\mathbb{S}^n_{++} \subset \mathbb{S}^n$ denote the sets of positive semidefinite and positive definite matrices, respectively. 

For $A,B \in \mathbb{S}^n$, we write $A \succeq B$ if $A-B \in \mathbb{S}^n_+$, and $A \succ B$ if $A-B \in \mathbb{S}^n_{++}$. In particular, $A \succeq 0$ (resp. $A \succ 0$) means that $A$ is symmetric positive semidefinite (resp. symmetric positive definite), i.e., $x^\top A x \ge 0$ (resp. $>0$) for all $x \in \mathbb{R}^n \setminus \{0\}$. Define
\begin{align*}
    \Sigma_{++}^m \coloneqq \Big\{ R:[0,T] \to \mathbb{S}^m_{++} \;\big|\; \exists \rho>0, \; \text{s.t. } R(t) \succeq \rho I_m, \text{ for all } t \in [0,T] \Big\},
\end{align*}
where $I_m$ denotes the $m\times m$ identity. Finally, for a matrix $A \in \mathbb{R}^{n \times m}$ we denote by $\|A\|_F$ the Frobenius norm, and for a vector $x\in\mathbb{R}^n$ we denote by $\|x\|_2$ the Euclidean norm. 

\subsection{Approximation of the Riccati operator}

The approximation of the Riccati operator by DeepONets is justified by combining regularity properties of the Riccati map with universal approximation results for OL. Under the assumptions in \eqref{eq:input}, the mapping $\mathcal F : \mathcal X \to \mathcal Y$ that associates the system parameters $\Theta$ with the solution $P(\cdot)$ of the Riccati equation is well-defined, continuous, and locally Lipschitz with respect to the input functions. This follows from standard sensitivity results for RDEs, which ensure continuous dependence of the solution on the coefficients. On the other hand, DeepONets are known to be universal approximators of nonlinear operators between Banach spaces. Therefore, the regularity of the Riccati operator together with the universal approximation property of DeepONets ensures that the solution operator $\mathcal F$ can be approximated arbitrarily well on compact subsets of $\mathcal X$.

\smallskip 
\noindent Let $\Theta = (A(\cdot),B(\cdot),Q(\cdot),R(\cdot),P_T)$, and define the admissible input space
\begin{equation}\label{eq:input}
    \mathcal{X} = \left\{\Theta \;\text{ s. t. }\; 
    \begin{array}{l} 
        \begin{array}{lll} A \in C([0,T]; \mathbb{R}^{n\times n}) & B \in C([0,T]; \mathbb{R}^{n\times m}) \\ Q \in C([0,T]; \mathbb{S}^n_+) & R \in C([0,T]; \Sigma_{++}^m) & P_T \in \mathbb{S}^n_{++} \end{array}
        \\ 
        \; (A(t),B(t)) \text{ is uniformly stabilizable} 
        \\ 
        \; (A(t),Q(t)^{1/2}) \text{ is uniformly detectable}
\end{array}\right\}.
\end{equation}

Let us recall that $(A,B)$ is stabilizable and $(A,Q^{1/2})$ is detectable if every eigenvalue $\lambda\in\CC$ of $A$ with non-negative real part satisfies 
\begin{align*}
    \text{rank}\Big(\lambda I -A\,|\, B\Big) = n \quad\text{ and }\quad \text{rank}\begin{pmatrix}\lambda I - A \\ Q^{1/2}\end{pmatrix} = n.  
\end{align*}

For all $\Theta\in\mathcal X$, the Riccati equation \eqref{eq:Riccati} admits a unique solution $P\in C([0,T];\mathbb S_{++}^n)$ \cite{sontag2013mathematical,anderson2007optimal}. Accordingly, we define the output space
\begin{align}\label{eq:output}
    \mathcal Y \coloneqq \Big\{P \in C([0,T];\mathbb S_{++}^n) \;\big|\; P \text{ solves } \eqref{eq:Riccati} \text{ for some } \Theta\in\mathcal X\Big\}.
\end{align}
This choice of spaces allows us to view the Riccati map as an operator 
\begin{align} \label{eq4}
    \mathcal F : \mathcal X  \mapsto \mathcal Y, \quad \mathcal F(\Theta)=P(\cdot).
\end{align}
which is continuous and locally Lipschitz with respect to the system parameters.

\subsection{Error propagation and stability analysis of the learned controller}

Consider system \eqref{eq:state} with quadratic cost \eqref{eq:functional}. For $\Theta=(A,B,Q,R,P_T)\in\mathcal X$, with $\mathcal X$ as in \eqref{eq:input}, let $P^\ast\in C([0,T];\mathbb S_{++}^n)$ be the unique solution of the DRE \eqref{eq:Riccati}, and denote $\NO P$ be an approximate solution of \eqref{eq:Riccati} obtained by DeepONet.  Define:
\begin{subequations}
	\begin{align}
        & K^*(t) = R^{-1}(t)B^\top(t)P^*(t), & & \NO K(t) = R^{-1}(t)B^\top(t)\NO P(t) & & \text{optimal feedback gain} \label{eq:feedback}
        \\
        & x^*(t), & & \NO x(t) & & \text{optimal state} \label{eq:optimal_state}
        \\
        & u^*(t) = -K^*(t)x^*(t), & & \NO u(t) = -\NO K(t)\NO x(t) & & \text{optimal control} \label{eq:control}    
    \end{align}
\end{subequations}

Denoting by $\mathcal{F}$ the exact Riccati operator and by $\NO{\mathcal{F}}$ its approximation, the learned control law is obtained by composing this operator with the feedback mapping. As a consequence, an approximation error at the operator level propagates through the following chain:
\begin{align*}
    \NO{\mathcal{F}} \sim \mathcal{F} \;\longrightarrow\; \NO K \sim K^\ast \;\longrightarrow\; \NO x \sim x^\ast \;\longrightarrow\; J(\NO u) \sim J(u^\ast).    
\end{align*}

The following result provides a rigorous quantification of this propagation mechanism, including explicit bounds on the trajectory deviation, cost suboptimality, and stability of the resulting closed-loop system.

\begin{theorem}\label{thm:error} 
Given $T>0$ and $\Theta=(A,B,Q,R,P_T)\in\mathcal X$, with $\mathcal X$ as in \eqref{eq:input}, let $P^\ast\in C([0,T];\mathbb S_{++}^n)$ be the unique solution of the DRE \eqref{eq:Riccati}. Let $\NO P$ be an approximate solution of \eqref{eq:Riccati} obtained by DeepONet, and define the error $E(t)\coloneqq \NO P(t) - P^\ast(t)$. Assume there exists $\varepsilon>0$ such that 
\begin{align}\label{eq:error_small}
        \sup_{t\in[0,T]} \|E(t)\| < \varepsilon.  
\end{align}
Then, for any initial datum $x_0\in\RR^n$, we have the following
\begin{itemize}
    \item[1.] \textbf{Finite}-\textbf{horizon performance error.} There exist constants $C_1, C_2, C_3, C_4 > 0$, depending on $\Theta$ and $T$, such that the optimal states and controls given by \eqref{eq:optimal_state} and \eqref{eq:control} satisfy
    \begin{align}
        & \sup_{t\in[0,T]} \|\NO x(t) - x^\ast(t)\|_2 \leq C_1 \|x_0\|_2 \Big(e^{C_2\varepsilon T}-1\Big), \label{eq:error_est}
        \\
        & J(\NO u) - J(u^\ast) \leq C_3\|x_0\|^2_2\varepsilon^2 Te^{C_4\varepsilon T}. \label{eq:J_est}
    \end{align}

    \item[2.] \textbf{Stability of the learned closed}-\textbf{loop system.} There exists $\varepsilon^\ast=\varepsilon^\ast(\Theta,T)>0$ such that if $\varepsilon<\varepsilon^\ast$ the approximate optimal state $\NO x(t)$ is uniformly stable on $[0,T]$. More precisely, there exist constants $M(\Theta,T),\mu(\Theta,T) > 0$ such that
    \begin{align*}
        \|\NO x(t)\|_2 \leq M e^{-\mu t}\|x_0\|_2, \quad \text{for all } t \in [0,T].    
    \end{align*}
    \end{itemize}
\end{theorem}

\begin{proof}[Proof sketch]

We provide here the main ideas for the proof of Theorem \ref{thm:error}. Complete details can be found in Appendix \ref{app:proof}.

The argument is based on the fact that the feedback gain depends linearly on the Riccati matrix. Hence the approximation error on the learned Riccati operator propagates directly to the closed-loop dynamics. First, from the definition \eqref{eq:feedback} of the feedback gains, we can estimate  
\begin{align*}
    \|\NO K(t)-K^\ast(t)\|_F \leq C\|E(t)\|_F.
\end{align*}
So controlling the Riccati error uniformly in time also controls the feedback perturbation. 

\medskip
\noindent For the trajectory estimate, let
\begin{align*}
    e(t)\coloneqq\NO x(t)-x^\ast(t), \quad A^\ast_{cl}(t)\coloneqq A(t)-B(t)K^\ast(t).
\end{align*}

Since $x^\ast$ solves the optimal closed-loop system and $\NO x$ solves the perturbed one, subtracting the two equations yields an error dynamics of the form
\begin{align*}
    \dot e(t)=A^\ast_{cl}(t)e(t)-B(t)R^{-1}(t)B^\top(t)E(t)\NO x(t), \quad e(0)=0.    
\end{align*}

Next, introduce the Lyapunov functional $V_e(t)\coloneqq e(t)^\top P^\ast(t)e(t)$. Because $P^\ast(t)\in \mathbb S_{++}^n$ uniformly on $[0,T]$, $V_e(t)$ is equivalent to $\|e(t)\|_2^2$. Differentiating $V_e(t)$ along the error dynamics, and using the Riccati equation satisfied by $P^\ast(t)$, the nominal terms cancel while the remaining terms are proportional to the perturbation $E(t)$. This gives a differential inequality of the form
\begin{align*}
    \dot V_e(t)\leq \alpha \|E(t)\|_F \sqrt{V_e(t)} + \beta\|E(t)\|_F V_e(t),    
\end{align*}
for suitable constants $\alpha,\beta>0$. An application of Gr\"onwall's lemma, together with $e(0)=0$, then yields \eqref{eq:error_est}.

\medskip 
For the cost estimate, we write $J(\NO u)-J(u^\ast)$ and expand the quadratic terms around the optimal pair $(x^\ast,u^\ast)$. Since the LQR cost is quadratic and $u^\ast$ is optimal, the first-order terms cancel, while the remaining terms are controlled by
\begin{align*}
   \|\NO x-x^\ast\|_{L^\infty(0,T)}^2 \quad\text{ and }\quad \|\NO u-u^\ast\|_{L^\infty(0,T)}^2.
\end{align*}
Using
\begin{align*}
    \NO u-u^\ast = -(\NO K-K^\ast)\NO x-K^\ast(\NO x-x^\ast),    
\end{align*}
together with the previous gain and trajectory bounds and a Lyapunov analysis, one obtains \eqref{eq:J_est}.

\medskip 
\noindent Finally, for the stability part, the nominal optimal closed-loop system 
\begin{align*}
    \dot x^\ast(t)=\big(A(t)-B(t)K^\ast(t)\big)x^\ast(t)    
\end{align*}
is uniformly exponentially stable under the standing stabilizability and detectability assumptions used to define the admissible class $\mathcal X$. Since
\begin{align*}
    \NO A_{cl}(t)-A^\ast_{cl}(t) =-B(t)\big(\NO K(t)-K^\ast(t)\big),    
\end{align*}
the learned closed-loop matrix is a small perturbation of the nominal exponentially stable one whenever the DeepONet approximation error is sufficiently small. Standard robustness of exponential stability for linear time-varying systems then yields the existence of $\varepsilon^\ast>0$ such that, if \eqref{eq:error_small} holds with $\varepsilon<\varepsilon^\ast$, the learned closed-loop system remains uniformly exponentially stable.
\end{proof}

Theorem \ref{thm:error} shows that the learned controller inherits the key qualitative properties of the optimal solution, provided the approximation error of the Riccati operator is sufficiently small. In particular, the result establishes robustness of exponential stability with respect to operator perturbations, together with quantitative bounds on the induced performance degradation. This goes beyond standard approximation guarantees by providing a control-theoretic interpretation of the learned operator, ensuring that its use within a feedback loop remains reliable.

\section{Numerical experiments}\label{sec:experiments}

In this section, we present numerical simulations to demonstrate the effectiveness of DeepONets in solving LQR problems. Following our discussion in Section \ref{sec:deeponet}, both time-independent and time-dependent parameters will be considered.

All simulations in this paper were conducted on a laptop running Windows 11 Home (version 24H2, OS build 26100.7840), equipped with an Intel\textregistered\ Core\texttrademark\ Ultra 7 255HX CPU (20 cores and 20 threads) and 32 GB of RAM. The system model is an HP OMEN MAX Gaming Laptop 16-ah0xxx, which includes an NVIDIA GeForce RTX 5060 Laptop GPU with 4 GB of dedicated memory. The implementation was developed in Python 3.13.9 using PyTorch 2.8.0+cu129.

\subsection{Time-independent parameters}

We shall focus here on the case of time-independent matrices $(A,B,Q,R)$, for which the DRE \eqref{eq:Riccati} reduces to the ARE \eqref{eq:RiccatiAlgebraic}. Our objective is to learn the solution operator $\mathcal F: (A,B,Q,R) \mapsto P$. 

\subsubsection{3-dimensional case}

In a first round of experiments, we have focused on a simple 3-dimensional case. As we commented before, to learn properly $\mathcal F$ and achieve good generalization performances, the quality and coverage of the training dataset are crucial. 

The LQR weighting matrices $Q$ and $R$ are chosen diagonal to penalize states and control inputs independently. For what concerns the dynamics matrices $A$ and $B$, instead, to construct representative training samples we exploit the Brunovsky canonical form \eqref{eq:brunovsky1}-\eqref{eq:brunovsky2}.

\medskip
\noindent In dimension $n=3$, this reduces to the following three cases.

\medskip
\noindent\textbf{Case1. Single}-\textbf{input controllable form:}
\begin{equation*}
    A = \begin{bmatrix} 0 & 1 & 0 \\ 0 & 0 & 1 \\ -a_0 & -a_1 & -a_2 \end{bmatrix}, \quad 
    B = \begin{bmatrix} 0 \\ 0 \\ 1 \end{bmatrix}, \quad
    Q = \begin{bmatrix} q_1 & 0 & 0 \\ 0 & q_2 & 0 \\ 0 & 0 & q_3 \end{bmatrix}, \quad
    R = \begin{bmatrix} r_1  \end{bmatrix}.
\end{equation*}

\medskip
\noindent\textbf{Case 2. Decoupled mixed}-\textbf{order form:}
\begin{equation*}
    A = \begin{bmatrix} 0 & 1 & 0 \\ -a_0 & -a_1 & 0 \\ 0 & 0 & -b_0 \end{bmatrix}, \quad 
    B = \begin{bmatrix} 0 & 0 \\ 1 & 0 \\ 0 & 1 \end{bmatrix}, \quad
    Q = \begin{bmatrix} q_1 & 0 & 0 \\ 0 & q_2 & 0 \\ 0 & 0 & q_3 \end{bmatrix}, \quad
    R = \begin{bmatrix} r_1 & 0 \\ 0 & r_2 \end{bmatrix}.
\end{equation*}

\medskip
\noindent \textbf{Case 3. Fully decoupled first}-\textbf{order form:}
\begin{equation*} 
    A = \begin{bmatrix} -a_0 & 0 & 0 \\ 0 & -b_0 & 0 \\ 0 & 0 & -c_0 \end{bmatrix}, \quad
    B = \begin{bmatrix} 1 & 0 & 0 \\ 0 & 1 & 0 \\ 0 & 0 & 1 \end{bmatrix}, \quad 
    Q = \begin{bmatrix} q_1 & 0 & 0 \\ 0 & q_2 & 0 \\ 0 & 0 & q_3 \end{bmatrix}, \quad
    R = \begin{bmatrix} r_1 & 0 & 0 \\ 0 & r_2 & 0 \\ 0 & 0 & r_3 \end{bmatrix}.
\end{equation*}

To ensure a comprehensive and representative dataset, we generated sufficient random samples for each of the above structural types. In particular, for each configuration, system parameters are randomly sampled from uniform distributions:
\begin{align*}
    a_i \sim \mathcal{U}(-2,2), \quad q_j, r_j \sim \mathcal{U}(0.1,1.1).    
\end{align*}
Moreover, the coefficients $a_i$ are sampled to cover the following spectral characterization of $A$:
\begin{itemize}
    \item[1.] Fully stable: all eigenvalues have negative real parts.
    \item[2.] Fully unstable: all eigenvalues have positive real parts.
    \item[3.] Partially stable: some eigenvalues have negative real parts and others positive.
\end{itemize}

A total of 15000 samples are generated per trial, with an 80\%–20\% train-test split. The ground-truth solutions of the ARE are computed using a standard CARE solver. The DeepONet model we employed consists of 
\begin{itemize}
    \item a branch network with $5$ layers of hidden dimensions $36$, $512$, $256$, $128$, and $256$, respectively;
    \item a trunk network with $4$ layers of hidden dimensions $2$, $128$, $256$, and $256$, respectively. 
\end{itemize}
Training has been performed for $1500$ epochs using the Adam optimizer. 

\medskip
\noindent\textbf{Performance evaluation.} We assess the quality of the learned operator through two complementary metrics, reflecting both control performance and approximation accuracy:
\begin{itemize}
    \item[1.] \textbf{closed}-\textbf{loop stability}, evaluating whether the feedback law induced by the learned operator preserves stability of the controlled system, which is the key property from a control-theoretic perspective;
    \item[2.] \textbf{test loss}, measuring the discrepancy between the predicted and true Riccati solutions, and providing a quantitative indicator of operator approximation accuracy.
\end{itemize}

This dual evaluation allows us to simultaneously quantify how well the operator is approximated and how reliably the resulting controller behaves when deployed in closed loop.

Table \ref{table1} summarizes the prediction performance over $10$ independent trials. The results demonstrate a high level of robustness of the learned controller: in $7$ out of $10$ trials, all test samples are successfully stabilized, while in the reaming $3$ trials only a single unstable case is observed. This corresponds to a stability rate exceeding 99.99\%, indicating that the learned operator preserves the qualitative behavior of the optimal feedback law across a wide range of systems. The test loss remains consistently low, confirming that this robustness is achieved together with accurate operator approximation.

\begin{table}[!h]
    \centering
    \begin{tabular}{cccccc}
        \toprule
        \textbf{Trial No.} & \shortstack{\textbf{Stable} \\ \textbf{samples}} & \shortstack{\textbf{Unstable} \\ \textbf{samples}} & \shortstack{\textbf{Stable} \\ \textbf{eigenvalues}} & \shortstack{\textbf{Unstable} \\ \textbf{eigenvalues}} & \textbf{Test loss} \\
        \midrule
        $1$  & $3000$  & $0$  & $9000$  & $0$ & $4.93\times 10^{-3}$ \\
        $2$  & $3000$  & $0$  & $9000$  & $0$ & $3.62\times 10^{-3}$ \\
        $3$  & $2999$  & $1$  & $8999$  & $1$ & $6.21\times 10^{-3}$ \\
        $4$  & $3000$  & $0$  & $9000$  & $0$ & $4.95\times 10^{-3}$ \\
        $5$  & $3000$  & $0$  & $9000$  & $0$ & $4.74\times 10^{-3}$ \\
        $6$  & $3000$  & $0$  & $9000$  & $0$ & $5.70\times 10^{-3}$ \\
        $7$  & $2999$  & $1$  & $8999$  & $1$ & $5.11\times 10^{-3}$ \\
        $8$  & $3000$  & $0$  & $9000$  & $0$ & $4.37\times 10^{-3}$ \\
        $9$  & $2999$  & $1$  & $8999$  & $1$ & $5.60\times 10^{-3}$ \\
        $10$ & $3000$  & $0$  & $9000$  & $0$ & $9.50\times 10^{-3}$ \\
        \midrule
        Mean & & & & & $5.47\times 10^{-3}$ \\
        Std  & & & & & $1.58\times 10^{-3}$ \\
        \bottomrule
    \end{tabular}
    \caption{DeepONet prediction results over 10 independent trials in dimension $n=3$ (test set: 3000 samples per trial).}
    \label{table1}
\end{table}

These results show that the learned operator yields a feedback law that is robust to approximation errors, preserving closed-loop stability in virtually all cases while maintaining good quantitative accuracy.

To further investigate the failure mechanism, we focus on Trial~3, which contains the only unstable sample. Figure~\ref{Eigenvalue1} presents the eigenvalue distributions of the true and predicted closed-loop systems for this trial. 
\begin{figure}[!h]
    \begin{minipage}{0.45\textwidth}
        \centering
        \includegraphics[width=\textwidth]{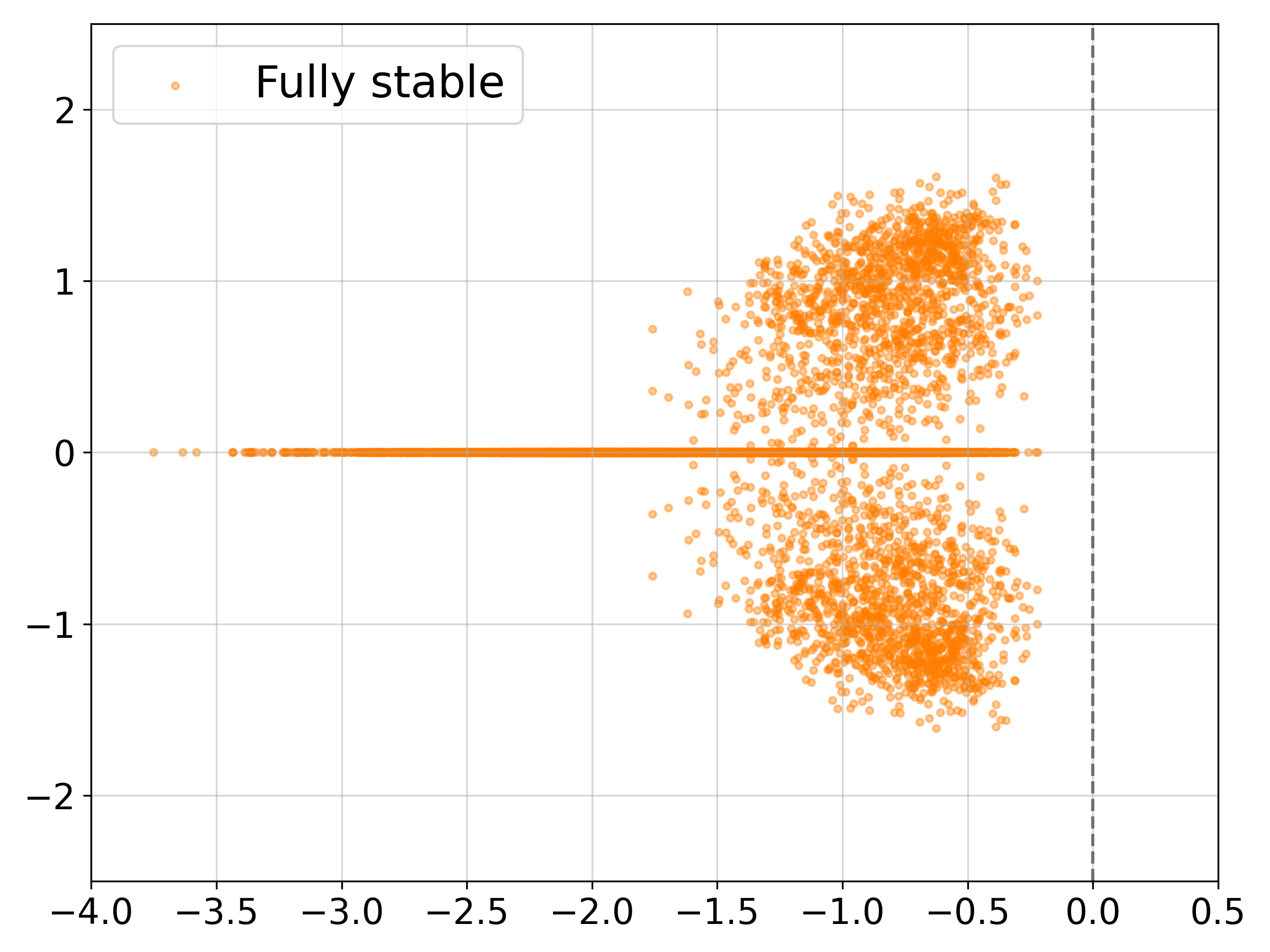}
    \end{minipage}
    \hfill
    \begin{minipage}{0.45\textwidth}
        \centering
        \includegraphics[width=\textwidth]{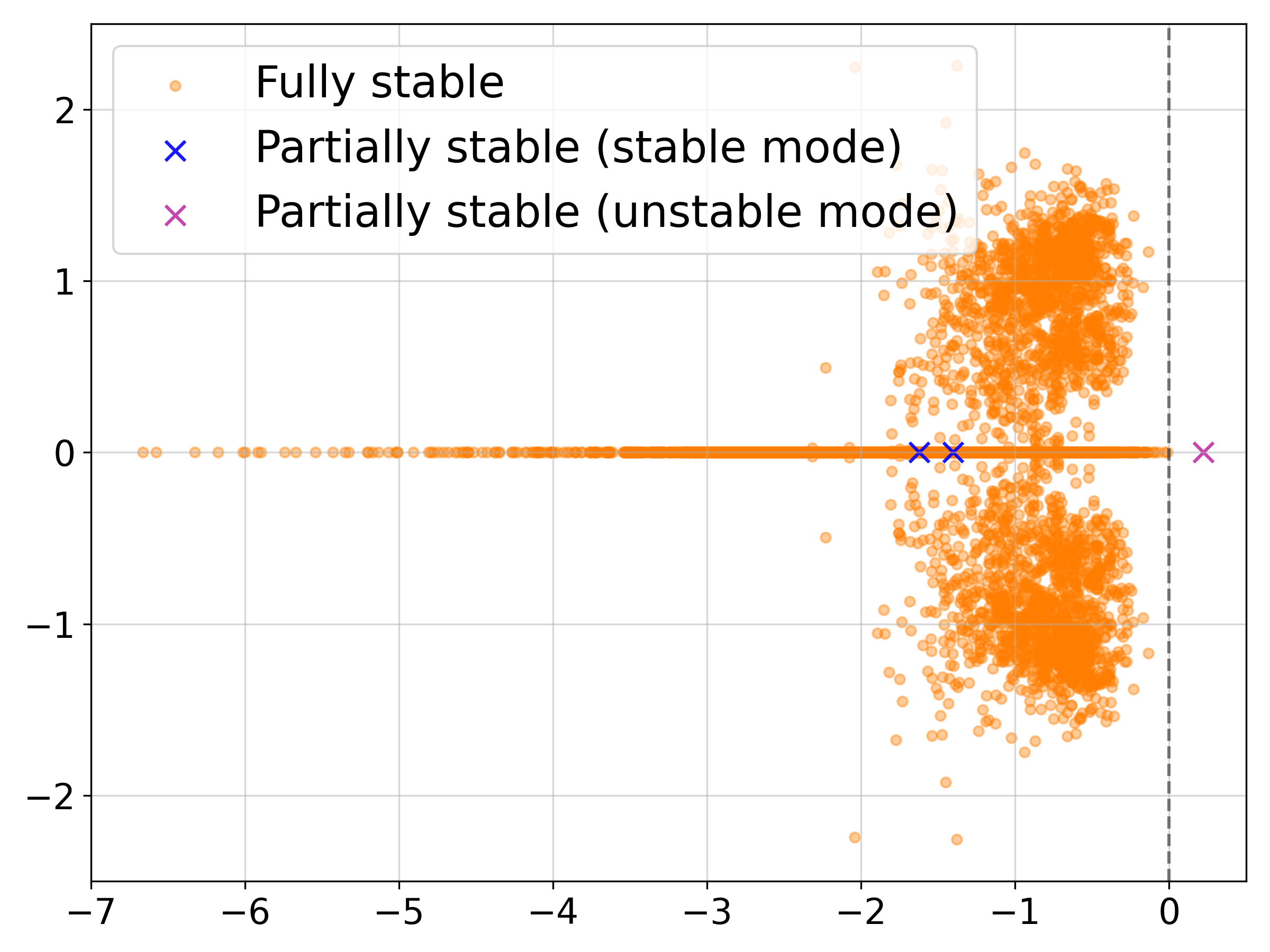}
    \end{minipage}
    \caption{Eigenvalue distribution for 3-d LQR problems. On the left we have the true closed-loop matrix $A^{\text{cl}}$ and on the right the DeepONet approximated one $\NO{A}^{\text{cl}}$.}\label{Eigenvalue1}
\end{figure}

Figure~\ref{Eigenvalue1} show the eigenvalue distributions of $A^{\text{cl}}$ and $\NO{A}^{\text{cl}}$, respectively. Under the true control law, all eigenvalues lie strictly in the left half-plane, confirming closed-loop stability. For the predicted control law, the eigenvalues closely match the true distribution and remain in the left half-plane for almost all samples. However, one eigenvalue is observed to shift slightly into the right half-plane, resulting in the only unstable case.

A closer examination reveals that, for this sample, the approximation error $\|E\|$ exceeds the threshold $\varepsilon^\ast$ identified in Theorem \ref{thm:error}. This leads to the observed deviation in the eigenvalue location and explains the loss of stability. Importantly, this behavior provides an empirical validation of the theoretical prediction: instability occurs precisely when the conditions ensuring stability preservation are violated. In this sense, the failure case is not anomalous but rather confirms the sharpness and relevance of the theoretical bound.

Overall, the stable-sample prediction accuracy exceeds 99.99\%, while the rare failure case is fully explained by the theoretical error threshold. This confirms that the learned operator behaves in accordance with the control-theoretic guarantees established in Section \ref{sec:math}.

\subsubsection{Progressive DeepONet architecture: 4-dimensional case}

As mentioned in Section \ref{subsec:progressive}, a central challenge in OL is scalability with respect to the system dimension. To address this issue, we have proposed a Progressive OL strategy allowing to transfer knowledge learned in low-dimensional settings to higher-dimensional problems.

In this section, we assess the performance of the progressive DeepONet by transferring our model trained in dimension $n=3$ to the case $n=4$. 

Following the same data generation procedure as in the 3-d case, we consider five Brunovsky types (\([4]\), \([3,1]\), \([2,2]\), \([2,1,1]\), and \([1,1,1,1]\)) and generate samples covering stable, unstable, and mixed dynamics. A total of $19000$ training samples and $1000$ test samples are used. We compared two models under identical settings: 
\begin{itemize}
    \item[1.] the proposed \textbf{progressive DeepONet}, which only trains the embedding and extension layers; 
    \item[2.] a \textbf{standard 4}-\textbf{d DeepONet} trained from scratch with architecture \(64 \rightarrow 512 \rightarrow 256 \rightarrow 128 \rightarrow 256\) (branch) and \(2 \rightarrow 128 \rightarrow 256 \rightarrow 256\) (trunk).
\end{itemize} 
Both models are trained for $3000$ epochs using Adam with learning rate $10^{-4}$.

As before, we evaluated the performances using closed-loop stability and prediction accuracy. The standard 4-d DeepONet achieves a stable-sample accuracy of $97.80\%$, while the progressive DeepONet achieves $97.50\%$. At the eigenvalue level, the accuracy is $99.42\%$ and $99.38\%$, respectively. 

The marginal performance gap indicates that a substantial portion of the complexity of the Riccati operator can be captured in a low-dimensional representation learned in the $3$-d setting. In this sense, the progressive strategy provides empirical evidence that the operator admits a structured representation in which its essential features can be encoded in a reduced latent space and transferred across dimensions. From a control perspective, these results further confirm that the learned operator retains robustness properties when transferred across dimensions. Despite the dimensional increase and the reduced number of trainable parameters, the progressive architecture preserves closed-loop stability with high reliability, indicating that the essential stability-inducing structure of the Riccati operator is captured in the learned representation.

At the same time, Table 2 shows that the progressive DeepONet significantly reduces computational cost, with a reduction of approximately 95\% in the number of trainable parameters and about 77\% in training time. These gains, obtained with only a marginal loss in stability accuracy, further support the interpretation that the Riccati operator may exhibit an intrinsic low-dimensional structure in the considered regimes, that can be effectively exploited through a latent representation.

\begin{table}[h!]
    \centering    
    \begin{tabular}{lccc}
        \hline
        \textbf{Aspect} & \textbf{4-d DeepONet} & \textbf{Progressive DeepONet} & \textbf{Reduction} 
        \\
        \hline
        Parameter count & $494.8$k & $22.4$k & $95.47\%$  
        \\
        Training time & $271.81$s & $61.81$s & $77.25\%$ 
        \\
        Stable sample accuracy & $97.80\%$ & $97.50\%$ & 
        \\
        Stable eigenvalue accuracy & $99.42\%$ & $99.38\%$ & 
        \\
        Test loss & $4.44\times 10^{-2}$ & $8.74\times 10^{-2}$ & 
        \\
        \hline
    \end{tabular}
    \caption{Performance and computational resource comparison for $4$-dimensional systems.}\label{table3}
\end{table}

These results show that the progressive architecture achieves accuracy comparable to a standard DeepONet trained from scratch, while using significantly fewer trainable parameters and requiring substantially less training time. More precisely, the progressive model retains a high level of closed-loop stability and prediction accuracy, with only a marginal degradation compared to the baseline model. At the same time, it reduces the number of trainable parameters by an order of magnitude and accelerates convergence, demonstrating a clear advantage in terms of computational efficiency.

Overall, the proposed progressive DeepONet provides an efficient and scalable framework for approximating high-dimensional Riccati solution operators. Beyond computational efficiency, the results suggest that the operator may exhibit structural regularities across dimensions that can be captured through a reduced latent representation, even though a rigorous theoretical justification of this phenomenon remains an open question.

\subsubsection{10-dimensional case}

To further evaluate the scalability of the proposed OL framework, we extend our experiments to 10-d systems. 

Let us stress that, when $n=10$, there are $42$ possible Brunovsky structural configurations of the matrix $A$. Generating and training a dataset covering all such cases is computationally prohibitive and unnecessary for our purposes. Instead, we chose to focus on the following particular Brunovsky canonical form
\begin{equation*}
    A = \begin{bmatrix}
        0 & 1 & 0 & \cdots & 0 \\
        0 & 0 & 1 & \cdots & 0 \\
        \vdots & \vdots & \vdots & \ddots & \vdots \\
        0 & 0 & 0 & \cdots & 1 \\
        -a_0 & -a_1 & -a_2 & \cdots & -a_{9}
    \end{bmatrix}, 
    \quad B = \begin{bmatrix} 0 \\ 0 \\ \vdots \\ 0 \\ 1\end{bmatrix}, 
    \quad Q=\begin{bmatrix} q_1 & 0 & 0 & \cdots & 0 \\ 0 & q_2 & 0 & \cdots & 0 \\ \vdots & & 0 & \cdots & \vdots \\ 0 & 0 & \cdots & q_9 & 0 \\ 0 & 0 & 0 & \cdots & q_{10}\end{bmatrix}, 
    \quad R=[r_1], 
\end{equation*}
with $q_i\geq 0$ and $r_1>0$. The coefficients $a_i$ are once again randomly sampled to generate systems with stable, unstable, and mixed dynamics.

Due to the increased computational cost in high dimensions, we construct a moderately sized dataset consisting of 6000 training samples (3000 stable and 3000 unstable) and 200 test samples (100 stable and 100 unstable). A compact DeepONet architecture is adopted, with branch network $(211\to256\to64$ and trunk network $2\to128\to64$, trained for 1000 epochs using Adam with learning rate $10^{-4}$.

We evaluate performance using the same stability-based criteria as in the 3-d case. Figure~\ref{Eigenvalue2} presents the eigenvalue distributions and Table~\ref{table4} summarizes the results on the 10-d test set. The model achieves a stable-sample accuracy of $98.50\%$, while the eigenvalue-level accuracy reaches $99.70\%$.

Notably, all unstable predictions correspond to systems that do not satisfy the assumptions underlying Theorem \ref{thm:error}, in particular the conditions ensuring stability of the nominal closed-loop system. This observation is consistent with the theoretical framework, as stability preservation is guaranteed only within the admissible class $\mathcal{X}$. Outside this regime, no such guarantee is expected, and the observed behavior aligns with this limitation.

\begin{figure}[!h]
    \begin{minipage}{0.45\textwidth}
        \centering
        \includegraphics[width=\textwidth]{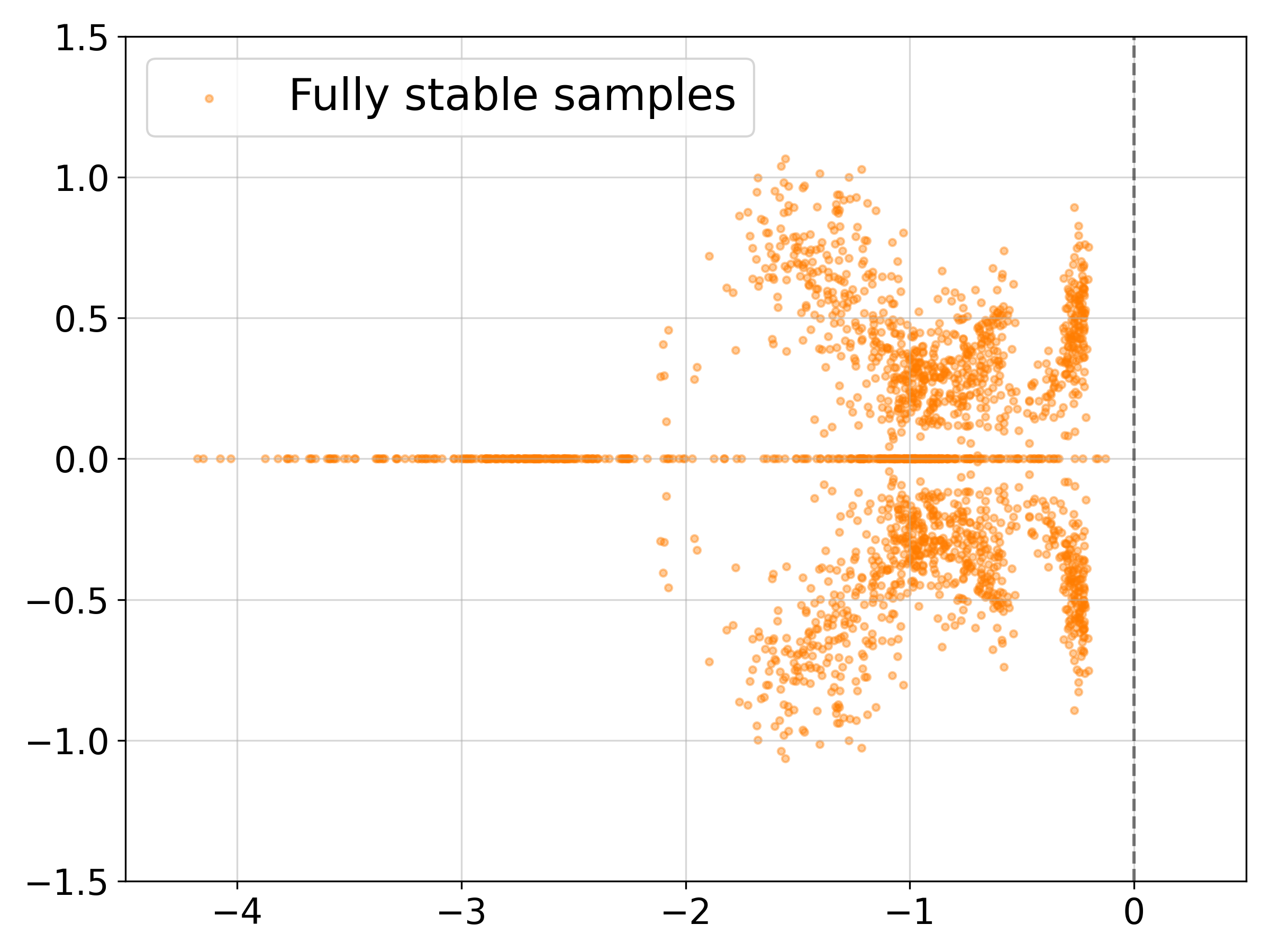}
    \end{minipage}
    \hfill
    \begin{minipage}{0.45\textwidth}
        \centering
        \includegraphics[width=\textwidth]{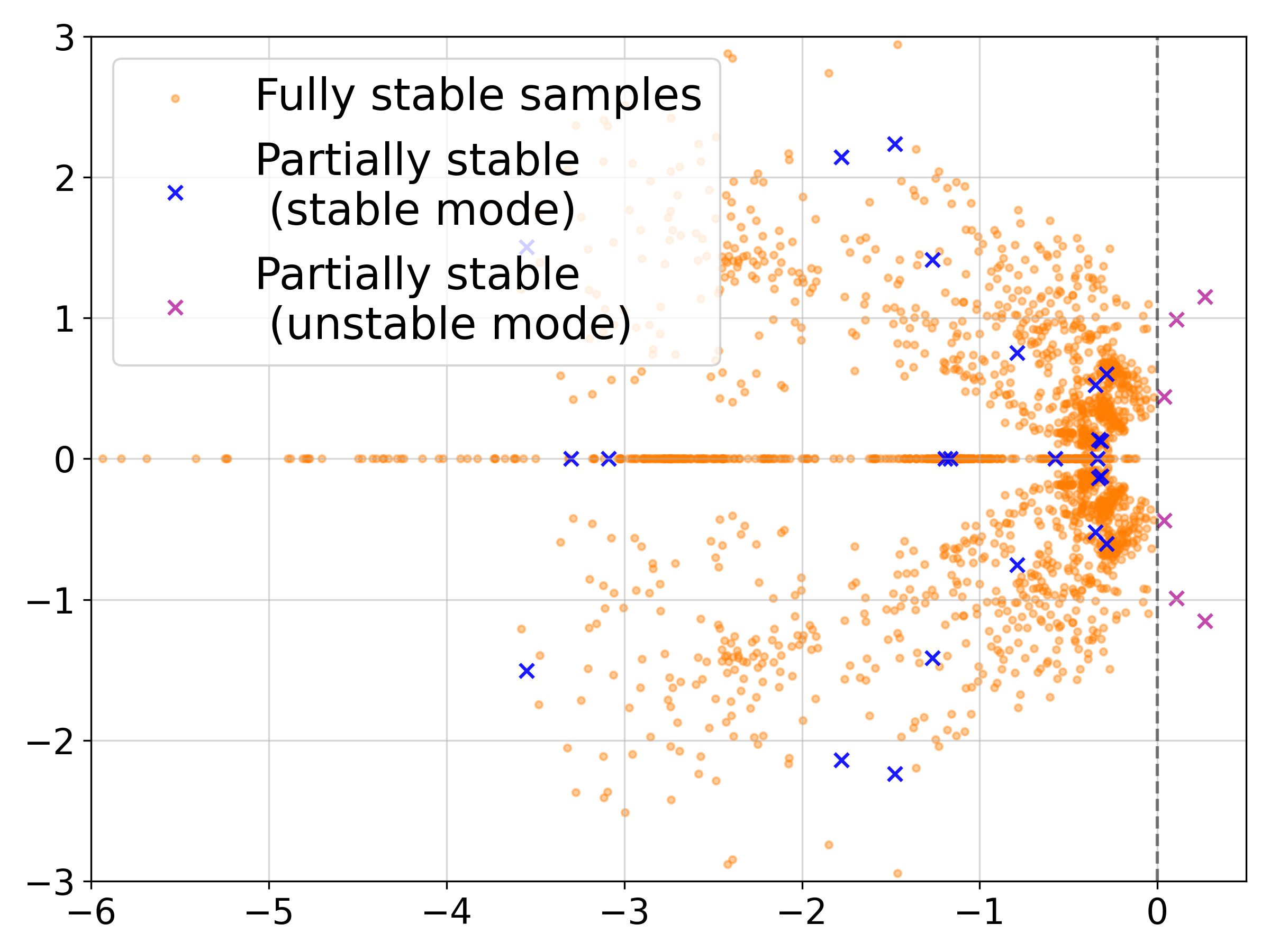}
    \end{minipage}
    \caption{Eigenvalue distribution for 10-d LQR problems. On the left we have the true closed-loop matrix $A^{\text{cl}}$ and on the right the DeepONet approximated one $\NO{A}^{\text{cl}}$.}\label{Eigenvalue2}
\end{figure}

\begin{table}[htbp]
    \centering    
    \begin{tabular}{ccccc}
        \toprule
        \shortstack{\textbf{Stable} \\ \textbf{samples}} & \shortstack{\textbf{Unstable} \\ \textbf{samples}} & \shortstack{\textbf{Stable} \\ \textbf{eigenvalues}} & \shortstack{\textbf{Unstable} \\ \textbf{eigenvalues}} & \textbf{Test loss} \\
        \midrule
        $197$ & $3$ & $1994$ & $6$ & $8.77\times10^{-3}$ \\       
        \bottomrule
    \end{tabular}
    \caption{Summary of DeepONet prediction results for 10-d (test set: 200 samples).}
    \label{table4}
\end{table}

Overall, our numerical results consistently support the theoretical framework developed in Section \ref{sec:math}. Across all tested regimes, the learned operator preserves the stability properties of the optimal feedback law with very high probability, even in the presence of non-negligible approximation errors. Moreover, the rare failure cases occur precisely when the theoretical assumptions are violated or when the approximation error exceeds the predicted threshold. This provides strong empirical evidence that the proposed approach not only achieves accurate operator approximation, but also ensures reliable control performance in a robust and theoretically predictable manner.

\subsection{Time-dependent parameters}
In this section, we evaluate the proposed DeepONet framework for learning the solution operator of time-varying DREs:
\begin{align*}
    \mathcal{F}:\big(A(\cdot), B(\cdot), Q(\cdot), R(\cdot), P_T\big) \mapsto P(\cdot),
\end{align*}
where $P(t)$ denotes the unique symmetric positive definite solution of \eqref{eq:Riccati}. As for the time-independent case before, we have performed simulations for 3- and 10-dimensional models.

\subsubsection{Data generation, DeepONet architecture and experimental settings}
For a system with state dimension $n$ and control dimension $m$, the interval $[0,1]$ is discretized into $N_t$ uniform time steps. The time-varying matrices $A(t)$, $B(t)$, and $Q(t)$ are parameterized via a finite trigonometric basis as in \eqref{eq:trigonometric}.

All coefficients $A_0, B_0, M_0, C_{ki}, D_{ki}, E_{ki}$ with $k=1,2$ and $i= 1,\dots, r_{\text{base}}$ are sampled from $\mathcal{N}(0, 0.05^2)$. Here we have set $r_{\text{base}}=2$. Moreover, to ensure that our dataset fulfills the necessary conditions for the existence of a LQR control, we applied a filtering procedure and retain only those matrices that satisfy stabilizability of $(A(t),B(t))$ and detectability of $(A(t), Q^{1/2}(t))$ at all discrete time steps. The control weighting matrix $R$, instead, is assumed to be constant: $R(t) = R = I_m$ for all $t\in[0,1]$. 

For the 3-d case, a dataset of 1000 valid trajectories is obtained with the above procedure. For the 10-d case, instead, our dataset is constituted by 3000 valid trajectories.

The terminal condition is fixed as $P_T = I_n$. Each trajectory is numerically solved via a backward fourth-order Runge-Kutta scheme (RK4). 

As for the architecture, for the 3-d case we considered a DeepONet with branch and trunk networks having inner dimensions $[256,256,128]$ and $[128,128,128]$, respectively, both with GELU activation. The model is trained for $60$ epochs using Adam with learning rate $1\times10^{-3}$. 

For the 10-d case, instead, we considered a DeepONet whose branch and trunk networks have inner dimensions $[1024,1024,512,512]$ and $[256,512,512]$, respectively, again both with GELU activation. The model is trained for $150$ epochs using Adam with learning rate $10^{-4}$. 

Finally, to quantify prediction accuracy for the closed-loop system, we define the following error metrics. 

\begin{displaymath}
    \begin{array}{ll}
        \text{Riccati solution errors}
        \\
        \displaystyle e_P = \frac 1K \sum_{k=1}^K \left(\int_0^T \| P^{(k)}(t) - \NO{P}^{(k)}(t) \|_F^2 \, dt \right)^{\frac 12} & \quad \text{absolute error} 
        \\[15pt]
        \displaystyle e_P^r = \frac 1K \sum_{k=1}^K \left( \int_0^T \| P^{(k)}(t) \|_F^2 \, dt \right)^{-\frac 12} \left( \int_0^T \| P^{(k)}(t) - \NO{P}^{(k)}(t) \|_F^2 \, dt \right)^{\frac 12} & \quad \text{relative error}
        \\[25pt]     
        \text{State errors} 
        \\
        \displaystyle e_x = \frac 1K \sum_{k=1}^K \left( \int_0^T \| x^{(k)}(t) - \NO{x}^{(k)}(t) \|_2^2 \, dt \right)^{\frac 12} & \quad \text{absolute error} 
        \\[15pt]
        \displaystyle e_x^r = \frac 1K \sum_{k=1}^K \left( \int_0^T \| x^{(k)}(t) \|_2^2 \, dt \right)^{-\frac 12}\left( \int_0^T \| x^{(k)}(t) - \NO{x}^{(k)}(t) \|_2^2 \, dt \right)^{\frac 12} & \quad \text{relative error}
        \\[25pt]
        \text{Cost functional errors} 
        \\
        \displaystyle e_J = \frac 1K \sum_{k=1}^K \left| J^{(k)} - \NO{J}^{(k)} \right| & \quad \text{absolute error} 
        \\[15pt]
        \displaystyle e_J^r = \frac 1K \sum_{k=1}^K \frac{\left| J^{(k)} - \NO{J}^{(k)} \right|}{\left| J^{(k)} \right|} & \quad \text{relative error}
    \end{array}    
\end{displaymath}

Here, $K$ denotes the number of test trajectories. For each trajectory indexed by $k=1,\dots,K$, $P^{(k)}(t)$, $x^{(k)}(t)$, and $J^{(k)}$ represent the ground-truth Riccati solution, state trajectory, and cost functional, respectively, while $\NO{P}^{(k)}(t)$, $\NO{x}^{(k)}(t)$, and $\NO{J}^{(k)}$ denote their corresponding DeepONets predictions. 

\subsubsection{Results}

To evaluate the generalization capability of the trained model, we generate $K=100$ test trajectory from a random instance not seen during training. The results in Table \ref{error_compare} show that DeepONet accurately predicts $P(t)$, $x(t)$, and $J$, with errors slightly increasing from 3-d to 10-d but remaining reasonably small, confirming good generalization across dimensions.

\begin{table}[t]
    \centering
    \begin{tabular}{c c c}
        \toprule
        Error metric & $d=3$ & $d=10$ \\
        \midrule

        Test loss & $4.58\times10^{-6}$ & $2.310\times10^{-5}$ \\

        \midrule 
        
        $e_P$ & $2.92\times10^{-2}$ & $2.72\times10^{-1}$ \\
        $e_P^r$ & $1.66\times10^{-2}$ & $8.19\times10^{-2}$ \\
        
        \midrule
        
        $e_x$  & $1.59\times10^{-4}$ & $6.76\times10^{-3}$ \\
        $e_x^r$ & $1.06\times10^{-4}$ & $2.22\times10^{-3}$ \\
        
        \midrule
        
        $e_J$ & $9.26\times10^{-6}$ & $2.67\times10^{-3}$ \\
        $e_J^r$ & $2.90\times10^{-6}$ & $2.48\times10^{-4}$ \\
        
        \bottomrule
    \end{tabular}
    \caption{Average absolute and relative errors for $P(t)$, state $x(t)$, and cost $J$ in different dimensions.}
    \label{error_compare}
\end{table}

        
        

       

Moreover, Figure \ref{fig3d} illustrates the state trajectories $x(t)$ and corresponding control inputs $u(t)$ for a representative 3-d test system under both the true Riccati and DeepONet-based controllers. The results demonstrate that the learned operator accurately reproduces the system behavior and feedback control in both settings, confirming its effectiveness and generalization capability. 

\begin{figure}[!h]
	\centering
	\begin{minipage}{0.85\textwidth}
		\centering
		\includegraphics[width=\textwidth]{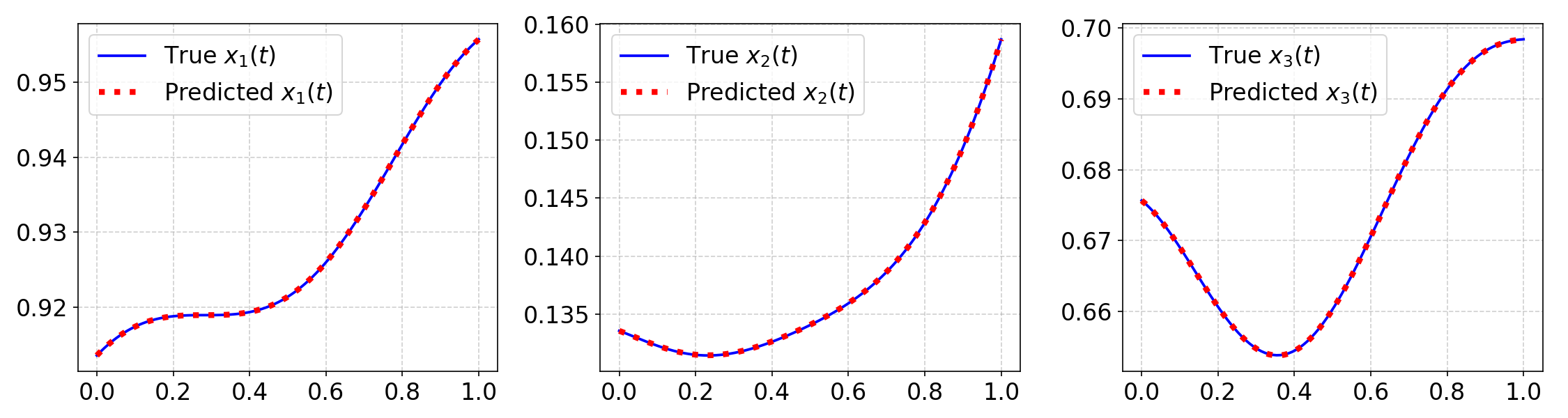}
	\end{minipage}
    
	\begin{minipage}{0.4\textwidth}
		\centering
		\includegraphics[width=\textwidth]{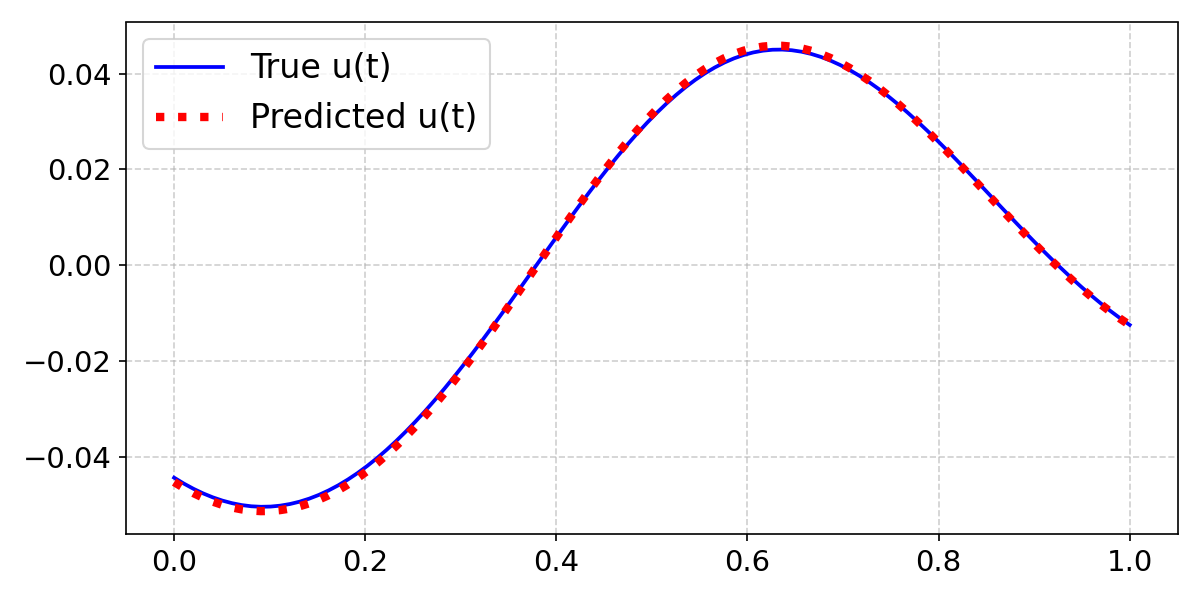}
	\end{minipage}
	\caption{Comparison of predicted and true trajectories (top) and predicted and tru controls (bottom) for a 3-d test system.}
	\label{fig3d}
\end{figure}




\subsection{Computational complexity and efficiency}

We conclude the numerical analysis by providing a unified assessment of the computational cost of the proposed approach, integrating the results obtained for both time-invariant (ARE) and time-dependent (DRE) settings. The objective is to quantify the cost of training the operator, the efficiency of its evaluation, and the resulting computational gain with respect to classical Riccati solvers.

Classical numerical methods for solving Riccati equations, such as CARE solvers in the time-invariant case and backward integration schemes for the differential Riccati equation, rely on dense linear algebra operations, including matrix factorizations and multiplications. As a consequence, their computational cost scales cubically with the state dimension, that is, of order $\mathcal{O}(n^3)$ per instance. When optimal control problems must be solved repeatedly for varying system parameters, the total cost therefore grows linearly with the number of instances, which becomes a major limitation in high-dimensional or real-time applications.

In contrast, the DeepONet-based framework follows an offline–online computational strategy. The offline phase consists of generating training data by solving multiple Riccati equations and optimizing the network parameters. This phase incurs a non-negligible one-time cost. Once the model is trained, however, the evaluation of the learned operator for a new parameter instance only requires a forward pass through the neural network. This inference step is computationally inexpensive and largely independent of the complexity of the underlying Riccati equation.

The magnitude of the offline training cost depends on both the problem dimension and the nature of the system parameters. The training time for all the configurations considered in our experiments a re collected in Table \ref{tab:offline_cost}.

\begin{table}[t]
    \centering
    \begin{tabular}{ccc}
        \toprule
        Learned solution & Dimension & Training time (s) \\
        \midrule
        
        \multirow{4}{*}{ARE}
        & $n=3$  & $42.63$ \\
        & $n=4$  & $271.81$ \\
        & $n=4$ progressive  & $61.81$ \\
        & $n=10$ & $382.82$ \\
        
        \midrule

        \multirow{2}{*}{DRE}
        & $n=3$  & $804.59$ \\
        & $n=10$ & $29951.34$ \\        
       
        \bottomrule
    \end{tabular}
    \caption{Offline training time for all the DeepONet models considered in our experiments.}
    \label{tab:offline_cost}
\end{table}

While in the case of the ARE the training time is relatively small, for the DRE the cost increases significantly due to the time-dependent nature of the problem, reaching a pick of approximately $8.3$ hours in dimension $n=10$. These results illustrate the higher computational burden associated with learning time-dependent operators, while remaining tractable within an offline setting.

Once trained, the computational advantage of the approach becomes evident at inference time. As reported in Tables \ref{time_compare_ARE} and \ref{time_compare_DRE}, the evaluation of the learned operator is up to several hundred times faster than solving the Riccati equation using classical CARE or RK4 routines. 

\begin{table}[t]
    \centering
    \begin{tabular}{c c c c}
        \toprule
        Dimension & Method & Time (ms) & Speedup \\
        \midrule
        
        \multirow{2}{*}{$n=3$}
        & CARE      & $0.51632$  & $1\times$ \\
        & DeepONet  & $0.00137$ & $376.86\times$ \\
        
        \midrule

        \multirow{2}{*}{$n=4$}
        & CARE      & $0.56849$ & $1\times$\\
        & DeepONet  & $0.00197$ & $287.72\times$\\
        & DeepONet-Progressive  & $0.00157$ & $360.58\times$\\
        
        \midrule
        
        \multirow{2}{*}{$n=10$}
        & CARE      & $0.87370$  & $1\times$ \\
        & DeepONet  & $0.00150$ & $582.47\times$ \\
        
        \bottomrule
    \end{tabular}
    \caption{Comparison of computational efficiency for solving ARE using CARE and DeepONet}
    \label{time_compare_ARE}
\end{table}

\begin{table}[th]
    \centering
    \begin{tabular}{c c c c}
        \toprule
        Dimension & Method & Time (ms) & Speedup \\
        \midrule
        \multirow{2}{*}{$n=3$} 
        & RK4        & $11.1564$  & $1\times$ \\
        & DeepONet   & $0.0065$   & $1723\times$ \\
        \midrule
        \multirow{2}{*}{$n=10$} 
        & RK4        & $52.2292$  & $1\times$ \\
        & DeepONet   & $0.2484$   & $210\times$ \\
        \bottomrule
    \end{tabular}
    \caption{Comparison of computational efficiency for solving DRE using RK4 and DeepONet}
    \label{time_compare_DRE}
\end{table}

This speedup is consistent across dimensions and reflects the fundamental difference between the two approaches: while classical solvers recompute the solution from scratch for each instance, the DeepONet model directly evaluates a learned mapping.

Overall, these results show that the proposed OL framework provides a substantial computational advantage in scenarios involving multiple problem instances. By shifting the main computational effort to an offline training phase and enabling fast evaluation at inference, the method allows for scalable and efficient approximation of Riccati solution operators across a wide range of system configurations. 

\section{Conclusions}\label{sec:conclusions}

In this work, we introduced an OL framework for the efficient approximation of the solution map associated with the differential Riccati equation arising in finite-horizon LQR problems. By recasting the Riccati equation as a nonlinear operator mapping time-dependent system parameters to matrix-valued trajectories, we leveraged DeepONets to learn this mapping in a data-driven yet structure-aware manner.

From a theoretical standpoint, we established a rigorous link between the approximation error of the learned operator and the performance of the induced feedback controller. In particular, we derived quantitative bounds on the suboptimality gap and proved that exponential stability of the closed-loop system is preserved under sufficiently accurate operator approximation. These results provide a principled framework to assess the reliability of neural operator approaches in optimal control.

On the computational side, we designed tailored DeepONet architectures capable of handling matrix-valued, time-dependent outputs, and introduced a progressive learning strategy to address scalability across dimensions. Numerical experiments demonstrated that the proposed approach achieves high accuracy, excellent stability preservation, and substantial computational speedups compared to classical Riccati solvers, especially in high-dimensional or repeated-query regimes.

Overall, our results highlight the potential of OL as a viable paradigm for amortized optimal control, enabling real-time evaluation of feedback laws in complex and parametric settings. Beyond the specific LQR setting, this work supports a broader perspective in which OL provides a systematic way to construct reusable surrogates for nonlinear differential equations arising in control and dynamical systems. This approach has the potential to significantly reduce computational cost in applications requiring repeated evaluation of parametrized models.

The results presented in this work open several avenues for further investigation at the interface of OL and optimal control.

\begin{itemize}
    \item[1. ] A first fundamental question concerns the generalization properties of neural operators in control settings. While our analysis provides a posteriori bounds linking the approximation error of the learned Riccati operator to the performance of the induced controller, a deeper understanding of generalization remains largely open. In particular, deriving a priori estimates on the sample complexity required to achieve a prescribed control accuracy, and characterizing how such requirements scale with the system dimension, the time horizon, and the regularity of the system coefficients, constitute important directions for future research.   
    \item[2.] Another relevant challenge is the design of structure-preserving learning architectures. The solution of the Riccati equation enjoys intrinsic properties such as symmetry and positive definiteness, which play a crucial role in ensuring stability of the closed-loop system. In the present work, these properties are only approximately satisfied by the learned operator, and violations may occur when the approximation error exceeds the threshold identified in Theorem \ref{thm:error}. Developing neural operator architectures that enforce such structural constraints by construction, for instance through suitable parameterizations of symmetric positive definite matrices or geometry-aware network designs, is a key open problem.
    \item[3.] Extending the proposed framework beyond the linear-quadratic setting also represents a natural and challenging direction. In particular, applying OL techniques to nonlinear optimal control problems would require approximating the solution operators of Hamilton–Jacobi–Bellman equations, which are fully nonlinear and may lack regularity. Similarly, incorporating state and control constraints, as in model predictive control, would significantly increase the complexity of the associated operators and raise new theoretical and computational issues.
    \item[4.] From the perspective of scalability, the progressive OL strategy introduced in this work provides an effective empirical approach to transfer knowledge across dimensions. However, its theoretical justification remains limited. An important open question is whether the Riccati solution operator admits a low-dimensional or dimension-independent representation, possibly in terms of reduced-order structures or compact embeddings of the parameter space. Establishing such properties would provide a rigorous foundation for scalable OL in high-dimensional control problems.
\end{itemize}

\section*{Acknowledgments}

The authors wish to thank Prof. Enrique Zuazua (Friedrich-Alexander-Universit\"at Erlangen-N\"urnberg, University of Deusto, and Universidad Aut\'onoma de Madrid) for valuable discussions and insightful suggestions that have significantly contributed to this work.

\appendix
\section{Proof of Theorem \ref{thm:error}}\label{app:proof}

\begin{proof} We prove the claims separately.

\medskip
\noindent\textbf{Trajectory error.} Define the optimal closed-loop matrix
\begin{align}\label{eq:closed_loop}
    A_{cl}^\ast(t)\coloneqq A(t) - B(t)K^\ast(t).
\end{align}
Using that
\begin{align*}
    \NO K(t) &= R^{-1}(t)B^\top(t)\NO P(t) = R^{-1}(t)B^\top(t) P^\ast(t) + R^{-1}(t)B^\top(t)\big(\NO P(t)-P^\ast(t)\big) 
    \\
    &= K^\ast(t) + R^{-1}(t)B^\top(t) E(t),    
\end{align*}
we can readily check that the error $e(t)\coloneqq \NO x(t) - x^\ast(t)$ follows the dynamics
\begin{align*}
    \begin{cases}
        \dot e(t) = A_{cl}^\ast(t)e(t) - B(t)R^{-1}(t)B^\top(t) E(t)\NO x(t), & t\in[0,T]
        \\
        e(0)=0. 
    \end{cases}
\end{align*}

Consider the Lyapunov function $V_e(t)\coloneqq e(t)^\top P^*(t)e(t)$. Differentiating along the solutions of \eqref{eq:Riccati} gives
\begin{align*}
    \dot V_e(t) &= \dot e^\top(t) P^\ast(t) e(t) +  e^\top(t) \dot P^\ast(t) e(t) + e^\top(t) P^\ast(t)\dot e(t) 
    \\
    &= e^\top(t) \dot P^\ast(t) e(t) + 2 e^\top(t) P^\ast(t) \dot e(t) 
    \\
    &= e^\top(t) \dot P^\ast(t) e(t) + 2 e^\top(t) P^\ast(t) \Big(A_{cl}^\ast(t) e(t) - B(t)R^{-1}(t)B^\top(t) E(t) \NO x(t)\Big) 
    \\
    &= e^\top(t) \dot P^\ast(t) e(t) + 2 e^\top(t) P^\ast(t) A_{cl}^\ast(t) e(t) - 2 e^\top(t) P^\ast(t) B(t)R^{-1}(t)B^\top(t) E(t) \NO x(t).
\end{align*}
Moreover, \eqref{eq:Riccati} and \eqref{eq:closed_loop} yield 
\begin{align*}
    & e^\top(t) \dot P^\ast(t) e(t)=e^\top(t) \Big(-A^\top(t) P^\ast(t)-P^\ast(t) A(t)+P^\ast(t) B(t)R^{-1}(t)B^\top(t) P^\ast(t)-Q(t)\Big)e(t), 
    \\
    & 2 e^\top(t) P^\ast(t) A_{cl}^\ast(t) e(t) = 2 e^\top(t) P^\ast(t)\Big(A(t)-B(t)R^{-1}(t)B^\top(t) P^\ast(t)\Big)e(t),
\end{align*}
so that we obtain 
\begin{align*}
    \dot V_e(t) = - e^\top(t) \Big(Q(t) + (K^\ast)^\top(t) R(t) K^\ast(t)\Big) e(t) - 2 e^\top(t) P^\ast(t) B(t)R^{-1}(t)B^\top(t) E(t)\NO x(t).
\end{align*}
Since $Q + (K^\ast)^\top R K^\ast \succeq 0$ uniformly on $[0,T]$, it follows that
\begin{align}\label{eq:V_dot_est}
    \dot V_e(t) \leq 2 \|P^\ast(t)\|_F \|B(t)\|_F^2 \|R^{-1}(t)\|_F \|E(t)\|_F \|e(t)\|_2 \|\NO x(t)\|_2. 
\end{align}
Now, by continuity of $P^\ast$ on $[0,T]$, there exist two positive constants $0<\gamma_\text{min}<\gamma_\text{max}$ such that
\begin{align*}
    \gamma_\text{min} I_n \preceq P^\ast(t) \preceq \gamma_\text{max} I_n.
\end{align*}
Hence $V_e \geq \gamma_\text{min} \|e\|_2^2$, which implies 
\begin{align}\label{eq:e_est}
    \|e\|_2 \leq \sqrt{\frac{V_e}{\gamma_\text{min}}}.    
\end{align}
Moreover, since $x^\ast$ is uniformly bounded on $[0,T]$, i.e., there exist constants $M_0>0$ such that
\begin{align*}
    \|x^\ast(t)\|_2 \leq M_0 \|x_0\|_2.
\end{align*}
From $\NO x(t)=x^\ast(t)+e(t)$ we therefore have 
\begin{align*}
    \|\NO x(t)\|_2 \leq \|x^\ast(t)\|_2+\|e(t)\|_2\leq M_0 \|x_0\|_2+\|e(t)\|_2.
\end{align*}
In view of this, we get from \eqref{eq:V_dot_est} that
\begin{align*}
    \dot V_e(t) \leq c_1 \|E(t)\|_F \|e(t)\|_2 \Big(M_0 \|x_0\|_2+\|e(t)\|_2\Big),
\end{align*}
where we have defined
\begin{align}\label{c1}
    c_1\coloneqq 2 \sup_{t\in[0,T]} \Big(\|B(t)\|^2_F \|R^{-1}(t)\|_F \|P^\ast(t)\|_F \Big).
\end{align}
Then, we obtain from \eqref{eq:e_est} that
\begin{align*}
    \dot V_e(t) &\leq c_1 \|E(t)\|_F \sqrt{\frac{V_e(t)}{\gamma_\text{min}}} \left(M_0 \|x_0\|_2+\sqrt{\frac{V_e(t)}{\gamma_\text{min}}}\right) \leq \alpha \|E(t)\|_F \sqrt{V_e(t)}+\beta \|E(t)\|_F V_e(t),
\end{align*}
where we have set 
\begin{align}\label{eq:alpha_beta}
    \alpha=\frac{c_1 M_0 \|x_0\|_2}{\sqrt{\gamma_\text{min}}} \quad\text{ and }\quad \beta=\frac{c_1}{\gamma_\text{min}}.    
\end{align}
This yields
\begin{align*}
    \frac{d}{dt} \sqrt{V_e(t)} =\frac{\dot V_e(t)}{2\sqrt{V_e(t)}} &\leq \frac{1}{2\sqrt{V_e(t)}} \left(\alpha \|E(t)\|_F \sqrt{V_e(t)}+\beta \|E(t)\|_F V_e(t)\right) 
    \\
    &=\frac{\alpha}{2}\|E(t)\|_F+\frac{\beta}{2}\|E(t)\|_F \sqrt{V_e(t)}.
\end{align*}
Setting $y(t)= \sqrt{V_e(t)}$, the above estimate becomes 
\begin{align*}
    \dot{y}(t)\leq \frac{\alpha}{2}\|E(t)\|_F+\frac{\beta}{2}\|E(t)\|_F y(t).
\end{align*}
Using the Gr\"onwall's lemma, combined with $y(0)=0$, we then obtain
\begin{align*}
    y(t) &\leq \frac{\alpha}{2}\int_0^t \exp\left(\frac{\beta}{2}\int_s^t\|E(\tau )\|_F\,d\tau\right)\|E(s)\|_F\,ds \leq \frac{\alpha \varepsilon}{2} \int_0^t e^{\frac{\beta}{2} \varepsilon (t-s)} \, ds
    \\
    &= \frac{\alpha}{\beta} \left( e^{\frac{\beta \varepsilon t}{2}} - 1 \right)=M_0\|x_0\|_2\sqrt{\gamma_\text{min}}\left( e^{\frac{\beta \varepsilon t}{2}} - 1 \right),
\end{align*}
with $\varepsilon$ given by \eqref{eq:error_small}. Then we get from \eqref{eq:e_est} and the definition \eqref{eq:alpha_beta} of $\beta$
\begin{align*}
    \|e(t)\|_2 \leq \sqrt{\frac{V_e(t)}{\gamma_\text{min}}}=\frac{y(t)}{\sqrt{\gamma_\text{min}}}\leq M_0\|x_0\|_2\left( e^{\frac{\beta \varepsilon t}{2}} - 1 \right) = M_0\|x_0\|_2\left( e^{\frac{c_1 \varepsilon t}{2\gamma_\text{min}}} - 1 \right),
\end{align*}
that is,  
\begin{align}\label{error}
    \sup_{t\in[0,T]}\|\NO x(t)-x^\ast(t)\|_2 \leq M_0\|x_0\|_2\left( e^{\frac{c_1 \varepsilon T}{2\gamma_\text{min}}} - 1 \right)
\end{align}
with $c_1$ given by \eqref{c1}.

\medskip 
\noindent\textbf{Cost error.} Consider now the Lyapunov function $V(t)\coloneqq \NO x(t)^\top P^\ast(t)\NO x(t)$. Differentiating along the trajectory $\NO x(t)$, and since $P^\ast(t)$ is symmetric, we get
\begin{align*}
    \dot V(t) &= \dot{\NO x}(t)^\top  P^\ast(t) \NO x(t) + \NO x(t)^\top \dot P(t)^\ast \NO x(t) + \NO x(t)^\top  P^\ast(t) \dot{\NO x}(t) 
    \\
    &= \NO x(t)^\top \dot P(t)^\ast \NO x(t) + 2 \NO x(t)^\top P^\ast(t)(A(t)\NO x(t)+B(t) \NO u(t)).
\end{align*}
The Riccati equation \eqref{eq:Riccati} then yields
\begin{align*} 
    \dot V(t) =&\, \NO x(t)^\top\Big(-A(t)^\top P^\ast(t)-P^\ast(t) A(t)+P^\ast(t) B(t)R^{-1}(t)B(t)^\top P^\ast(t)-Q(t)\Big)\NO x(t) 
    \\
    &+2\NO x(t)^\top P^\ast(t) A(t)\NO x(t) + 2\NO x(t)^\top P^\ast(t) B(t)\NO u(t).
\end{align*}
Since
\begin{align*}
    - &\NO x(t)^\top A(t)^\top P^\ast(t)\NO x(t) - \NO x(t)^\top P^\ast(t) A(t)\NO x(t) + 2\NO x(t)^\top P^\ast(t) A(t)\NO x(t) 
    \\
    &= -\Big(\NO x(t)^\top P^\ast(t) A(t)\NO x(t)\Big)^\top - \NO x(t)^\top P^\ast(t) A(t)\NO x(t) + 2\NO x(t)^\top P^\ast(t) A(t)\NO x(t) 
    \\
    &= - 2\NO x(t)^\top P^\ast(t) A(t)\NO x(t) + 2\NO x(t)^\top P^\ast(t) A(t)\NO x(t) = 0,
\end{align*}
we obtain from the previous identity that 
\begin{align*} 
    \dot V(t) = -\NO x(t)^\top Q(t)\NO x(t) + \NO x(t)^\top P^\ast(t) B(t)R^{-1}(t)B(t)^\top P^\ast(t)\NO x(t) + 2\NO x(t)^\top P^\ast(t) B(t)\NO u(t).
\end{align*}
Next, recalling that $K^\ast(t) =R^{-1}(t)B(t)^\top P^\ast(t)$, we have  
\begin{align*} 
    & (K^\ast(t))^\top R(t)K^\ast(t) =P^\ast(t) B(t)R^{-1}(t)B(t)^\top P^\ast(t) 
    \\
    & \NO u(t)^\top R(t)K^\ast(t)\NO x(t)=\NO x(t)^\top P^\ast(t) B(t)\NO u(t).
\end{align*}
Therefore
\begin{align*}
    \Big(\NO u(t) &+ K^\ast(t)\NO x(t)\Big)^\top R(t)\Big(\NO u(t)+K^\ast(t)\NO x(t)\Big) 
    \\
    &= \NO u(t)^\top R(t)\NO u(t) + 2\NO u(t)^\top R(t)K^\ast(t)\NO x(t) + \NO x(t)^\top (K^\ast(t))^\top R(t)K^\ast(t)\NO x(t) 
    \\
    &=\NO u(t)^\top R(t)\NO u(t) +2\NO x(t)^\top P^\ast(t) B(t)\NO u(t) + \NO x(t)^\top P^\ast(t) B(t)R^{-1}(t)B(t)^\top P^\ast(t)\NO x(t).
\end{align*}
Hence,
\begin{align*}
    \NO x(t)^\top & P^\ast(t) B(t)R^{-1}(t)B(t)^\top P^\ast(t)\NO x(t) + 2\NO x(t)^\top P^\ast(t) B(t)\NO u(t) 
    \\
    &= \Big(\NO u(t)+K^\ast(t)\NO x(t)\Big)^\top R(t)\Big(\NO u(t)+K^\ast(t)\NO x(t)\Big)-\NO u(t)^\top R(t)\NO u(t),
\end{align*}
and finally
\begin{align*}
    \dot V(t) = -\NO x(t)^\top Q(t)\NO x(t)-\NO u(t)^\top R(t)\NO u(t) + \Big(\NO u(t)+K^\ast(t)\NO x(t)\Big)^\top R(t)\Big(\NO u(t)+K^\ast(t)\NO x(t)\Big).
\end{align*}
Equivalently,
\begin{align*}
    \NO x(t)^\top Q(t)\NO x(t)+\NO u(t)^\top R(t)\NO u(t) = -\dot V(t) + \Big(\NO u(t)+K^\ast(t)\NO x(t)\Big)^\top R(t)\Big(\NO u(t)+K^\ast(t)\NO x(t)\Big)
\end{align*}
and integrating over $[0,T]$ yields
\begin{align*}
    \int_0^T & \Big(\NO x(t)^\top Q(t) \NO x(t) + \NO u(t)^\top R(t) \NO u(t)\Big)\, dt 
    \\
    &= V(0) - V(T) + \int_0^T \Big(\NO u(t) + K^\ast(t)\NO x(t)\Big)^\top R(t) \Big(\NO u(t) + K^\ast(t) \NO x(t)\Big)\,dt.
\end{align*}
Since $P^\ast(T)=P_T$, we then obtain
\begin{align*}
    J(\NO{u}) &= \NO x(T)^\top P_T \NO x(T) + \int_0^T \Big(\NO x(t)^\top Q(t) \NO x(t) + \NO u(t)^\top R(t) \NO u(t)\Big)\, dt 
    \\
    &= \NO x(T)^\top P_T \NO x(T) + V(0) - V(T) + \int_0^T \Big(\NO u(t) + K^\ast(t)\NO x(t)\Big)^\top R(t) \Big(\NO u(t) + K^\ast(t) \NO x(t)\Big)\,dt
    \\
    &= x_0^\top P^\ast(0)x_0 + \int_0^T \Big(\NO u(t) + K^\ast(t) \NO x(t)\Big)^\top R(t) \Big(\NO u(t) + K^\ast(t) \NO x(t)\Big)\, dt.
\end{align*}
On the other hand, for the optimal control $u^\ast(t)=-K^\ast(t)x^\ast(t)$, we have
\begin{align*}
    J(u^\ast) = x_0^\top P^\ast(0)x_0.
\end{align*}
Therefore,
\begin{align*}
    J(\NO u) - J(u^\ast) = \int_0^T \Big(\NO u(t) + K^\ast(t) \NO x(t)\Big)^\top R(t) \Big(\NO u(t) + K^\ast(t) \NO x(t)\Big)\, dt.
\end{align*}
But since $\NO u(t)=-\NO K(t)\NO x(t)$, it follows that 
\begin{align*}
    \NO u(t) + K^\ast(t) \NO x(t) = \Big(K^\ast(t) - \NO K(t)\Big)\NO x(t) = -R^{-1}(t) B(t)^\top E(t) \NO x(t),
\end{align*}
and thus
\begin{align*}
    J(\NO u) - J(u^\ast) &=  \int_0^T \NO x(t)^\top E(t)^\top B(t) R^{-1}(t) B(t)^\top E(t) \NO x(t) \, dt 
    \\
    &\leq  \int_0^T \|\NO x(t)\|_2^2 \|E(t)\|_F^2 \|B(t)\|_F^2 \|R^{-1}(t)\|_F \, dt.
\end{align*}
Moreover, from \eqref{error} we have that
\begin{align*}
    \|\NO x(t)\|_2 &\leq \|x^\ast(t)\|_2 + \|\NO x(t) - x^\ast(t)\|_2
    \\
    &\leq M_0 \|x_0\|_2+M_0\|x_0\|_2\left( e^{\frac{c_1 \varepsilon T}{2\gamma_\text{min}}} - 1 \right)=M_0\|x_0\|_2 e^{\frac{c_1 \varepsilon T}{2\gamma_\text{min}}}.
\end{align*}
Hence, using \eqref{eq:error_small}, we can conclude that
\begin{align*}
    J(\NO u) - J(u^\ast) \leq  C_3\varepsilon^2 Te^{C_4\varepsilon T}\|x_0\|^2_2,
\end{align*}
with
\begin{align*}
    C_3\coloneqq M_0^2\sup_{t\in[0,T]}\Big(\|B(t)\|_F^2 \|R^{-1}(t)\|_F\Big)\quad\text{ and }\quad C_4\coloneqq \frac{c_1}{\gamma_\text{min}}
\end{align*}

\medskip
\noindent\textbf{Stability of the closed}-\textbf{loop system.} Define the optimal and approximated closed-loop matrices 
\begin{align*}
    A_{cl}^\ast(t) = A(t)-B(t)K^\ast(t) \quad\text{ and }\quad \NO{A}_{cl}(t) = A(t)-B(t) \NO{K}(t),    
\end{align*}
and denote their difference
\begin{align*}
    \Delta A(t) \coloneqq \NO{A}_{cl}(t) - A_{cl}^\ast(t) = -B(t) R^{-1}(t) B(t)^\top E(t).
\end{align*}

Because $B\in C([0,T];\RR^{n\times m})$ and $R\in C([0,T];\Sigma_{++}^m)$, both $B$ and $R^{-1}$ are uniformly bounded on $[0,T]$. Therefore there exists a constant $C_0>0$, depending only on the problem data, such that
\begin{align*} 
    \|\Delta A(t)\|_F \leq C_0 \|E(t)\|_F, \quad\text{ for all } t\in[0,T],
\end{align*}
and we have from \eqref{eq:error_small}
\begin{align*} 
    \sup_{t\in[0,T]}\|\Delta A(t)\|_F \leq C_0\varepsilon.
\end{align*}

Moreover, since the nominal optimal closed-loop system $\dot x(t)=A_{cl}^\ast(t)x(t)$ is uniformly exponentially stable, there exist constants $M>0$ and $\mu_0>0$, depending only on the problem data, such that its solution operator $\Phi_\ast(t,s)$ satisfies
\begin{align}\label{eq:sol_operator}
    \|\Phi_\ast(t,s)\|_2\leq M e^{-\mu_0(t-s)}, \quad \text{ for all } 0\leq s\leq t\leq T.
\end{align}
Let now $\NO x$ solve the learned closed-loop system
\begin{align*} 
    \dot{\NO x}(t)=\NO A_{cl}(t)\NO x(t) = \Big(A_{cl}^\ast(t)+\Delta A(t)\Big)\NO x(t), \quad \NO x(0)=x_0.
\end{align*}
The variation-of-constants formula gives
\begin{align*} 
    \NO x(t) = \Phi_\ast(t,0)x_0 + \int_0^t \Phi_\ast(t,s)\Delta A(s)\NO x(s)\,ds.
\end{align*}
Using \eqref{eq:sol_operator} we therefore have
\begin{align*}
    \|\NO x(t)\|_2 \leq M e^{-\mu_0 t}\|x_0\|_2 + MC_0\varepsilon\int_0^t e^{-\mu_0(t-s)} \|\NO x(s)\|_2\,ds.
\end{align*}
Set
\begin{align*}
    y(t)\coloneqq e^{\mu_0 t}\|\NO x(t)\|_2.
\end{align*}
Multiplying the previous inequality by $e^{\mu_0 t}$, we obtain
\begin{align*}
    y(t) \leq M\|x_0\|_2 + MC_0\varepsilon \int_0^t y(s)\,ds.
\end{align*}
By Gr\"onwall’s inequality,
\begin{align*}
    y(t) \leq M\|x_0\|_2 e^{MC_0\varepsilon t}, \quad \text{ for all } t\in[0,T],
\end{align*}
that is, 
\begin{align*}
    \|\NO x(t)\|_2 \leq Me^{-(\mu_0 - MC_0\varepsilon)t}\|x_0\|_2, \quad \text{ for all } t\in[0,T].
\end{align*}
Finally, if we choose 
\begin{align*}
    \varepsilon^\ast\coloneqq \frac{\mu_\ast}{2M C_0},    
\end{align*}
then
\begin{align*}
    \mu\coloneqq \mu_0 - MC_0\varepsilon \geq \frac{\mu}{2}>0, 
\end{align*}
which immediately yields
\begin{align*}
    \|\NO x(t)\|_2 \leq M e^{-\mu t}\|x_0\|_2, \quad\text{ for all } t\in[0,T].
\end{align*}
\end{proof}

\nocite{*}    

\bibliographystyle{elsarticle-num} 
\bibliography{references} 

@book{abou2012matrix,
  title={{Matrix Riccati equations in control and systems theory}},
  author={Abou-Kandil, Hisham and Freiling, Gerhard and Ionescu, Vlad and Jank, Gerhard},
  year={2012},
  publisher={Birkh{\"a}user}
}

@article{alcala2018autonomous,
	title={{Autonomous vehicle control using a kinematic Lyapunov-based technique with LQR-LMI tuning}},
	author={Alcala, Eugenio and Puig, Vicen{\c{c}} and Quevedo, Joseba and Escobet, Teresa and Comasolivas, Ramon},
	journal={Control Eng. Pract.},
	volume={73},
	pages={1--12},
	year={2018},
	publisher={Elsevier}
}

@book{anderson2007optimal,
	title={{Optimal control: linear quadratic methods}},
	author={Anderson, Brian DO and Moore, John B},
	year={2007},
	publisher={Courier Corporation}
}

@book{bertsekas2012dynamic,
  title={{Dynamic programming and optimal control: Volume I}},
  author={Bertsekas, Dimitri},
  volume={4},
  year={2012},
  publisher={Athena scientific}
}

@book{bittanti1991riccati,
  title={The {R}iccati equation},
  author={Bittanti, Sergio and Laub, Alan and Willems, Janc},  
  year={1991},
  publisher={Springer-Verlag}
}

@article{choi1999lqr,
  title={{LQR design with eigenstructure assignment capability and application to aircraft flight control}},
  author={Choi, Jae Weon and Seo, Young Bong},
  journal={IEEE Trans. Aerosp. Electron. Syst.},
  volume={35},
  number={2},
  pages={700--708},
  year={1999},
  publisher={IEEE}
}

@misc{coddington1956theory,
  title={{Theory of ordinary differential equations}},
  author={Coddington, Earl A and Levinson, Norman and Teichmann, T},
  year={1956},
  publisher={American Institute of Physics}
}

@article{cui2025learning,
  title={{Learning-based adaptive optimal control of linear time-delay systems: A value iteration approach}},
  author={Cui, Leilei and Pang, Bo and Krsti{\'c}, Miroslav and Jiang, Zhong-Ping},
  journal={Automatica},
  volume={171},
  pages={111944},
  year={2025},
  publisher={Elsevier}
}

@article{deng2022approximation,
  title={{Approximation rates of DeepONets for learning operators arising from advection-diffusion equations}},
  author={Deng, Beichuan and Shin, Yeonjong and Lu, Lu and Zhang, Zhongqiang and Karniadakis, George Em},
  journal={Neur. Netw.},
  volume={153},
  pages={411--426},
  year={2022},
  publisher={Elsevier}
}

@article{dhingra2025modeling,
  title={{Modeling and LQR control of insect sized flapping wing robot}},
  author={Dhingra, Daksh and Kaheman, Kadierdan and Fuller, Sawyer B},
  journal={npj Robotics},
  volume={3},
  number={1},
  pages={6},
  year={2025},
  publisher={Nature Publishing Group UK London}
}

@article{elkhatem2022robust,
	title={{Robust LQR and LQR-PI control strategies based on adaptive weighting matrix selection for a UAV position and attitude tracking control}},
	author={Elkhatem, Aisha Sir and Engin, Seref Naci},
	journal={Alex. Eng. J.},
	volume={61},
	number={8},
	pages={6275--6292},
	year={2022},
	publisher={Elsevier}
}

@article{fong2018dual,
  title={{Dual-loop iterative optimal control for the finite horizon LQR problem with unknown dynamics}},
  author={Fong, Justin and Tan, Ying and Crocher, Vincent and Oetomo, Denny and Mareels, Iven},
  journal={Syst. Control Lett.},
  volume={111},
  pages={49--57},
  year={2018},
  publisher={Elsevier}
}

@article{gu2006generalized,
  title={{Generalized LQR control and Kalman filtering with relations to computations of inner-outer and spectral factorizations}},
  author={Gu, Guoxiang and Cao, Xi-Ren and Badr, Hesham},
  journal={IEEE Trans. Automat. Control},
  volume={51},
  number={4},
  pages={595--605},
  year={2006},
  publisher={IEEE}
}

@article{he2016dual,
  title={{Dual-mode nonlinear MPC via terminal control laws with free-parameters}},
  author={He, Defeng},
  journal={IEEE/CAA J. Autom. Sin.},
  volume={4},
  number={3},
  pages={526--533},
  year={2016},
  publisher={IEEE}
}

@article{kalabic2020constraint,
  title={{A constraint-separation principle in model predictive control}},
  author={Kalabi{\'c}, Uro{\v{s}} V and Kolmanovsky, Ilya V},
  journal={Automatica},
  volume={121},
  pages={109190},
  year={2020},
  publisher={Elsevier}
}

@book{kirk2004optimal,
  title={{Optimal control theory: an introduction}},
  author={Kirk, Donald E},
  year={2004},
  publisher={Courier Corporation}
}

@article{klemm2020lqr,
  title={{LQR-assisted whole-body control of a wheeled bipedal robot with kinematic loops}},
  author={Klemm, Victor and Morra, Alessandro and Gulich, Lionel and Mannhart, Dominik and Rohr, David and Kamel, Mina and de Viragh, Yvain and Siegwart, Roland},
  journal={IEEE Robot. Autom. Lett.},
  volume={5},
  number={2},
  pages={3745--3752},
  year={2020},
  publisher={IEEE}
}

@article{kumar2013robust,
  title={{Robust LQR controller design for stabilizing and trajectory tracking of inverted pendulum}},
  author={Kumar, E Vinodh and Jerome, Jovitha},
  journal={Procedia Eng.},
  volume={64},
  pages={169--178},
  year={2013},
  publisher={Elsevier}
}

@book{kwakernaak1972linear,
	title={{Linear optimal control systems}},
	author={Kwakernaak, Huibert and Sivan, Raphael},
	volume={1},
	year={1972},
	publisher={Wiley-interscience New York}
}

@book{lancaster1995algebraic,
  title={Algebraic {R}iccati equations},
  author={Lancaster, Peter and Rodman, Leiba},
  year={1995},
  publisher={Clarendon press}
}

@article{lanthaler2022error,
	title={{Error estimates for deeponets: A deep learning framework in infinite dimensions}},
	author={Lanthaler, Samuel and Mishra, Siddhartha and Karniadakis, George E},
	journal={Trans. Math. Appl.},
	volume={6},
	number={1},
	pages={1--141},
	year={2022},
	publisher={Oxford University Press}
}

@article{lu2021learning,
	title={{Learning nonlinear operators via DeepONet based on the universal approximation theorem of operators}},
	author={Lu, Lu and Jin, Pengzhan and Pang, Guofei and Zhang, Zhongqiang and Karniadakis, George Em},
	journal={Nature Mach. Intell.},
	volume={3},
	number={3},
	pages={218--229},
	year={2021},
	publisher={Nature Publishing Group UK London}
}

@article{mohammed2018active,
  title={{Active vibration control of cantilever beam by using optimal LQR controller}},
  author={Mohammed, Hadeer Abd UL-Qader and Wasmi, Hatem Rahem},
  journal={J. Eng.},
  volume={24},
  number={11},
  pages={1--17},
  year={2018}
}

@article{peng2025linear,
  title={{Linear quadratic regulator-based coordinated optimization for harmonic compensation in multifunctional grid-connected inverters}},
  author={Peng, Xianghua and Yin, Jingyuan and Yu, Jiancheng and Song, Jinzhao and Liu, Keyan and Li, Zhao and Wei, Tongzhen},
  journal={Int. J. Electr. Power Energy Syst.},
  volume={172},
  pages={111333},
  year={2025},
  publisher={Elsevier}
}

@inproceedings{quero2024physics,
  title={{Physics-informed machine learning for uav control}},
  author={Quero, Carlos Alexander Osorio and Martinez-Carranza, Jose},
  booktitle={2024 21st International Conference on Electrical Engineering, Computing Science and Automatic Control (CCE)},
  pages={1--6},
  year={2024},
  organization={IEEE}
}

@article{raissi2019physics,
  title={{Physics-informed neural networks: A deep learning framework for solving forward and inverse problems involving nonlinear partial differential equations}},
  author={Raissi, Maziar and Perdikaris, Paris and Karniadakis, George E},
  journal={J. Comput. Phys.},
  volume={378},
  pages={686--707},
  year={2019},
  publisher={Elsevier}
}

@article{reddy2024learning,
  title={{Learning-based optimal control of linear time-varying systems over large time intervals}},
  author={Reddy, Vasanth and Boker, Almuatazbellah and Eldardiry, Hoda},
  journal={Syst. Control Lett.},
  volume={185},
  pages={105750},
  year={2024},
  publisher={Elsevier}
}

@article{singh2015decentralized,
  title={{Decentralized control of oscillatory dynamics in power systems using an extended LQR}},
  author={Singh, Abhinav Kumar and Pal, Bikash C},
  journal={IEEE Trans. Power Syst.},
  volume={31},
  number={3},
  pages={1715--1728},
  year={2015},
  publisher={IEEE}
}

@article{slimane2025deep,
  title={{Deep reinforcement learning LQR controller design for MIMO systems applied to gas production facility}},
  author={Slimane, Kamel Ben and Tmar, Zied and Besbes, Mongi},
  journal={Int. J. Autom. Control},
  volume={19},
  number={6},
  pages={669--704},
  year={2025},
  publisher={Inderscience Publishers (IEL)}
}

@book{sontag2013mathematical,
 	title={{Mathematical control theory: deterministic finite dimensional systems}},
 	author={Sontag, Eduardo D},
 	volume={6},
 	year={2013},
 	publisher={Springer Science \& Business Media}
 }

@book{stevens2015aircraft,
  title={{Aircraft control and simulation: dynamics, controls design, and autonomous systems}},
  author={Stevens, Brian L and Lewis, Frank L and Johnson, Eric N},
  year={2015},
  publisher={John Wiley \& Sons}
}

@article{su2024lqr,
  title={{LQR-based control strategy for improving human--robot companionship and natural obstacle avoidance}},
  author={Su, Zefan and Yao, Hanchen and Peng, Jianwei and Liao, Zhelin and Wang, Zengwei and Yu, Hui and Dai, Houde and Lueth, Tim C},
  journal={Biomim. Intell. Robot.},
  volume={4},
  number={4},
  pages={100185},
  year={2024},
  publisher={Elsevier}
}

@book{trelat2005controle,
  title={{Contr{\^o}le optimal: th{\'e}orie \& applications}},
  author={Tr{\'e}lat, Emmanuel},
  volume={36},
  year={2005},
  publisher={Vuibert Paris}
}

@article{xu2018optimal,
  title={{Optimal control data scheduling with limited controller-plant communication}},
  author={Xu, Jiapeng and Wen, Chenglin and Xu, Daxing},
  journal={Sci. China Inf. Sci.},
  volume={61},
  number={1},
  pages={012202},
  year={2018},
  publisher={Springer}
}

@article{yu2026stochastic,
  title={{Stochastic Linear Quadratic Optimal Control for Continuous-Time Systems via Reinforcement Learning}},
  author={Yu, Jianglin and Wang, Bing-Chang and Meng, Deyuan},
  journal={Int. J. Robust Nonlinear Control},
  volume={36},
  number={4},
  pages={1876--1887},
  year={2026},
  publisher={Wiley Online Library}
}

@book{zhou1996robust,
  title={{Robust and optimal control}},
  author={Zhou, Kemin and Doyle, John Comstock and Glover, Keith and others},
  volume={40},
  year={1996},
  publisher={Prentice hall New Jersey}
}
\end{document}